\UseRawInputEncoding
\documentclass{amsart}[12pt]

\usepackage{mathrsfs,graphicx,latexsym,tikz,color,euscript,amsfonts,amssymb,amsmath,amscd,amsthm,epstopdf,hyperref,extarrows}

\usepackage[numbers,sort&compress]{natbib}

\theoremstyle{plain}
\newtheorem{thm}{Theorem}[section]
\newtheorem{defi}{Definition}[section]
\newtheorem{prop}{Proposition}[section]
\newtheorem{lem}{Lemma}[section]

\newtheorem{rem}{Remark}[section]

\newtheorem{claim}{Claim}

\numberwithin{equation}{section}

\newcommand{\CL}{\mathcal{L}}

\title[Equidistributed closed geodesics Problem]{Equidistributed Closed Geodesics on Closed Finsler and Riemannian Surfaces}

\begin{document}

\thanks{2020 \em Mathematics Subject Classification\em. 70H12, 53D25, 53D42. \\
\indent E-mail $\&$ Address: \textsuperscript{1}huiliu00031514@whu.edu.cn, School of Mathematics and Statistics, Wuhan University, China.
\textsuperscript{2}
liulei30@email.sdu.edu.cn, School of Mathematics, Shandong University, China.}

\maketitle

\begin{center}

\normalsize
Hui Liu\textsuperscript{1} and
Lei Liu\textsuperscript{2}
\bigskip

\end{center}

\date{}

\begin{abstract}
In this paper, we establish the existence of an equidistributed sequence of nondegenerate closed geodesics for generic Finsler, symmetric Finsler and Riemannian metrics on every closed surface. The proof relies on the volume property of embedded contact homology, established by Cristofaro-Gardiner, Hutchings and Ramos, along with specific local variational constructions and transversality arguments. Our approach is motivated by Irie's equidistribution result in \cite{Irie2018} for three-dimensional Reeb flows and the analogous result presented by Marques, Neves and Song \cite{MNS2017} for embedded minimal hypersurfaces.
\end{abstract}

\tableofcontents

\section{Introduction} \label{sec: introduction}

The present study investigates the distribution of closed geodesics on closed Finsler and Riemannian surfaces, which can be traced back to the work of Hadamard in the late 19th century \cite{Hadamard1898}. Using the Poincar\'e recurrence theorem, he demonstrated that the closed geodesics are densely distributed on any negatively curved closed Riemannian manifold. A dynamical proof was given by Anosov in his famous closing lemma, as the geodesic flow on negatively curved closed Riemannian manifold is also an ``Anosov" flow, (see \cite{Anosov1967} and \cite{KH1995}). For more revelent results on the density and equidistribution of closed geodesics, we refer to \cite{BBS1985},\cite{Bowen1972},\cite{CS2011},\cite{Zelditch1989} and the reference therein. For the open problems about geodesics, we refer to the survey \cite{BM2021}.

In general setting, the problem to the distribution of closed geodesics has been extensively studied for both Riemannian and Finsler manifolds. In 2015, by utilizing the volume property of embedded contact homology (ECH) established by Cristofaro-Gardiner, Hutchings and Ramos \cite{C-GHR2015}, Irie \cite{Irie2015} proved the density of Reeb orbits for a generic set of contact forms on every $3$-dimensional closed contact manifold, and established the density of closed geodesics for a generic set of Riemannian metrics on every closed surface. Inspired by the results of Marques, Neves and Song \cite{MNS2017} on the equidistribution of embedded minimal hypersurfaces, Irie \cite{Irie2018} discovered an equidistributed sequence of Reeb orbits for a generic set of contact forms on every $3$-dimensional closed contact manifold. It is well known that the Finsler and Riemannian geodesic flows are examples of Reeb flows. The space of Finsler metrics associates to an open subset in the space of contact forms, due to the fiberwise convexity. A conclusion of \cite{Irie2018} is that there exists an equidistributed sequence of closed geodesics for a generic set of Finsler metrics on every closed surface. However, due to the rigidity of symmetric Finsler and Riemannian metrics, the existence of an equidistributed sequence of closed geodesics for Riemannian and symmetric Finsler surfaces presents a more nuanced consideration. The Riemannian case, in particular, is more intricate because of the heightened rigidity exhibited by Riemannian metrics compared to Finsler metrics. Using the local perturbation of Finsler metrics in \cite{BI2016}, Chen \cite{Chen2019} proved the density of closed geodesics for a generic set of Finsler and symmetric Finsler metrics. Li and Staffa \cite{LS2022} proved the existence of an equidistributed sequence of closed geodesics for a generic set of Riemannian metrics on every closed surface using ``Weyl law" of the volume spectrum established by Liokumovich, Marques and Neves \cite{LMN2018}. In this paper, inspired by \cite{Irie2018} and \cite{MNS2017}, we further prove the existence of an equidistributed sequence of closed geodesics for a generic set of Finsler, symmetric Finsler and Riemannian metrics on every closed surface. This work not only provides a quantitative refinement of the generic density theorem in \cite{Chen2019}, but also generalizes the aforementioned result in \cite{LS2022} into the Finsler setting using an entirely distinct approach. 
Comparing to the Finsler case indicated from \cite{Irie2018}, our proof relies on a specific local variation of metrics and the transversality arguments in \cite{Irie2018} that adapted to the symmetric Finsler and Riemannian structure. In particular, the local variation is important in overcoming the challenges posted by the rigidity of symmetric Finsler and Riemannian metrics. This idea was also briefly and independently mentioned in \cite{LS2022} in 2022, while our paper was initially released on arXiv. Moreover, an analogous approach has been used in \cite{Pirprasad2022} and \cite{Prasad2023} to address the equidistributed periodic orbits on a closed oriented surface probably with boundary for a generic set of area-preserving diffeomorphisms.

\subsection{Main results}
We first refer to the following definition when discussing the concept of equidistribution.
\begin{defi}
Given a measure space $(X,\mathcal{A},\mu)$, where $X$ is a set, $\mathcal{A}$ is a $\sigma$-algebra and $\mu$ is a measure on $\mathcal{A}$ such that $\mu(X)$ is finite. A sequence of subsets $\{\Omega_i\}_{i\geq 1}\subset X$ is called equidistributed if the average integral of any smooth function $f$ over these subsets equals the average integral of $f$ over the entire space $X$. In another word, the sequence $\{\Omega_i\}_{i\geq 1}$ satisfies
$$
\lim_{N\rightarrow +\infty}\frac{1}{\sum_{i=1}^N \mu(\Omega_i)}\sum_{i=1}^N \int_{\Omega_i}fd\mu=\frac{1}{\mu(X)}\int_{X}f d\mu,\quad \forall f\in C^\infty(X,\mathbb{R}).
$$
\end{defi}

Before stating our main results, we first introduce some notations. Let $\Sigma$ be a closed and compact $2$-dimensional surface. Let $\mathcal F(\Sigma)$ and $\mathcal F^s(\Sigma)$ denote the collection of Finsler and symmetric Finsler metrics on $\Sigma$, respectively. In particular, $\mathcal F^s(\Sigma)$ is a subset of $\mathcal F(\Sigma)$. Let $\mathcal M(\Sigma)$ be the collection of Riemannian metrics on $\Sigma$, which forms a subset of $\mathcal F^s(\Sigma)$. For every Finsler metric $F$, we denote $\lambda_F$ the associated contact form on the unit cotangent bundle $S^*_F\Sigma$. If $F^2=g$ defines a Riemannian metric on $\Sigma$, we denote $\lambda_g:=\lambda_{F}$ the associated contact form on the unit cotangent bundle $S^*_g\Sigma$. Let $l_F$ be the length function on the $C^1$-loop space, and denote $l_g:=l_F$ if $F^2=g$ defines a Riemannian metric on $\Sigma$. Denote $\mathcal L^F$ the set of all Finsler $1$-currents on $(\Sigma,F)$, that is finite $\mathbb R_+$-linear formal combinations of closed geodesics with positive coefficients. Similarly, we denote $\mathcal L^g$ the set of all Riemannian $1$-currents on $(\Sigma,g)$. Given a Finsler or Riemannian $1$-current $C=\sum_i a_i \sigma_i$. For any differential $1$-form $\alpha$ on $S_F\Sigma$ and differential $1$-form $\beta$ on $S^*_F\Sigma$, we define $C(\alpha):=\sum_i a_i\int_{\tilde \sigma_i}\alpha$ and $C(\beta):=\sum_i a_i\int_{\gamma_{\sigma_i}}\beta$, where $\tilde \sigma\subset S_F\Sigma,\gamma_{\sigma}\subset S^*_F\Sigma$ are the corresponding curves of $\sigma$, see section~\ref{sub: Finsler and Rie setting}.
Moreover, in a given topological space, a certain property holds for generic elements in the topological space mean that these elements form a residual subset, i.e. it contains an intersection of countably many open and dense subsets.

Our first result concerns the existence of an ``almost" equidistributed sequences of $1$-currents for generic Finsler, symmetric Finsler and Riemannian metrics under the $C^\infty$-topology on every closed surface. The symmetric Finsler and Riemannian metrics are mainly considered.

\begin{thm}\label{thm: finsler}
For a generic $F\in \mathcal F^s(\Sigma)$, there exists a sequence of nondegenerate Finsler $1$-currents $\{C_N\}_{N\geq 1}\subset \CL^F$ such that
\begin{equation}\label{equ: limit of Reeb currents Finsler}
\lim_{N\rightarrow +\infty}C_N(f\lambda_F):=\int_{S^*_F\Sigma}f\lambda_F\wedge d\lambda_F,\quad \forall f\in C^\infty(\Sigma,\mathbb{R}),
\end{equation}
A similar statement holds for a generic Finsler metric $F\in \mathcal F(\Sigma)$.
\end{thm}

\begin{thm}\label{thm: riemann}
For a generic $g\in \mathcal M(\Sigma)$, there exists a sequence of nondegenerate Riemannian $1$-currents $\{C_N\}_{N\geq 1}\subset \CL^g$ such that
\begin{equation}\label{equ: limit of Reeb currents Riemannian}
\lim_{N\rightarrow +\infty}C_N(f\lambda_g)=\int_{S^*_g\Sigma}f\lambda_g\wedge d\lambda_g,\quad \forall f\in C^\infty(\Sigma,\mathbb{R}).
\end{equation}
\end{thm}

The validity of Theorem \ref{thm: finsler} - \ref{thm: riemann} extends to all closed surfaces. The proof relies on the utilization of ECH spectral invariants introduced in section \ref{sec: preliminaries}, along with the generic perturbation results in section \ref{sec: generic nondegenercy of Rimannian metrics} using transversality arguments. Note that the function $f$ is arbitrary on $\Sigma$, rather than $S^*_F\Sigma$. For a generic $F\in \mathcal F^s(\Sigma)$, one can use the same idea to improve \eqref{equ: limit of Reeb currents Finsler}  for every $f\in C^\infty(S^*_F\Sigma,\mathbb R)$ with symmetry $f(q,p)=f(q,-p),\forall (q,p)\in S^*_F\Sigma$. For generic Finsler metrics, the symmetric condition of $f$ can further be removed. Finally, using the method of approximation and selection in \cite{MNS2017}, we can construct a sequence of closed geodesics that are equidistributed on every closed surface for generic Finsler, symmetric Finsler and Riemannian metrics.

\begin{thm}\label{thm: equidis finsler}
For a generic $F\in \mathcal{F}^s(\Sigma)$, there exists a sequence of nondegenerate closed geodesics $\{\sigma_k\}_{k\geq 1}$ on $\Sigma$ that is equidistributed, satisfying
\begin{equation}\label{equ: limit of Finlser closed geodesics}
\lim_{k\rightarrow +\infty} \frac{\sigma_1(f\beta_F)+\cdots+\sigma_k(f\beta_F)}{l_F(\sigma_1)+\cdots+l_F(\sigma_k)}=\frac{\int_\Sigma fd\mathrm{vol}_F}{\mathrm{vol}_F\Sigma},\quad \forall f\in C^\infty(\Sigma,\mathbb{R}),
\end{equation}
where $\beta_F$ is the Hilbert form on $S_F\Sigma$, see \eqref{equ: def of Hilbert form}, and $d\mathrm{vol}_F$ is the Finsler volume form, see \eqref{equ: volume form}. A similar statement holds for generic $F\in\mathcal{F}(\Sigma)$.
\end{thm}

\begin{thm}\label{thm: equidis riemann}
For a generic $g\in\mathcal{M}(\Sigma)$, there exists a sequence of nondegenerate closed geodesics $\{\sigma_k\}_{k\geq 1}$ on $\Sigma$ that is equidistributed, satisfying
\begin{equation}\label{equ: limit of Riemannian closed geodesics}
\lim_{k\rightarrow +\infty} \frac{\sigma_1(fds_g)+\cdots+\sigma_k(fds_g)}{l_g(\sigma_1)+\cdots+l_g(\sigma_k)}=\frac{\int_\Sigma fd\mathrm{vol}_g}{\mathrm{vol}_g\Sigma},\quad \forall f\in C^\infty(\Sigma,\mathbb{R}),
\end{equation}
where $s_g$ represents the arc-length parameter on $\sigma$'s.
\end{thm}

Notice that the existence of an equidistributed sequence of closed geodesics for a generic $F\in\mathcal{F}(\Sigma)$ is a conclusion of \cite{Irie2018}, which can be reproved in the framework of geodesic flows. For symmetric Finsler and Riemannian metrics, the proofs are more delicate and does not follow from \cite{Irie2018}. Moreover, we refer to \cite{LS2022} for a different proof of Theorem \ref{thm: equidis riemann} using the ``Weyl law" of volume spectrum in the min-max theory established by Liokumovich, Marques and Neves \cite{LMN2018}.


\begin{rem}
The equidistributed sequences of closed geodesics $\{\sigma_k\}_{k\geq1}$ in Theorems \ref{thm: equidis finsler} and \ref{thm: equidis riemann} are selected from the sequence of $1$-currents $\{C_k\}_{k\geq 1}$ produced in Theorems \ref{thm: finsler} and \ref{thm: riemann} through proper iterations, where each $C_k$ exhibits an action that closely approximates to one of the $\mathrm{ECH}$ spectral invariants introduced in Proposition \ref{prop: properties of spectral invariant}. Additionally, for generic Finsler, symmetric Finsler and Riemannian metrics, the sequence $\{\sigma_k\}_{k\geq 1}$ can be chosen to be all nondegenerate even if the metrics are not bumpy, i.e. the closed geodesics are not all nondegenerate.
\end{rem}

This paper is organized as follows: In section \ref{sec: preliminaries}, we begin by introducing key properties of ECH spectral invariants. Subsequently, we present valuable relations between Reeb flows and Finsler (or Riemannian) geodesic flows. The Sard and Sard-Smale theorem are also mentioned in this section. In section \ref{sec: generic nondegenercy of Rimannian metrics}, we mainly prove Lemmas \ref{lem: genericity of nondegenercy 1} and Theorem \ref{lem: genericity of nondegenercy 2} for Finsler, symmetric Finsler and Riemannian metrics through a localized perturbation of Finsler metrics that adapt to the symmetric Finsler and Riemannian metrics, combined with some transversality techniques. Finally, in section \ref{sec: proof of main theorems}, we establish the proofs of Theorems \ref{thm: finsler} - \ref{thm: equidis riemann}, using the properties of ECH spectral invariants, the results in section \ref{sec: generic nondegenercy of Rimannian metrics} and the relations between Reeb flows and geodesic flows.

\vspace{3mm}
\textbf{Acknowledgments}. The authors would like to thank Professor Umberto Hryniewicz for his interest and Professor Gang Tian for his comments. The first author is supported by National Natural Science Foundation of China (NSFC) under grants No.12371195 and No.12022111. The second author is supported by the National Natural Science Foundation of China under grant No.12401238, the Natural Science Foundation of Shandong Province, China under grant No.ZR2024QA188.

\section{Preliminaries}\label{sec: preliminaries}

\subsection{Spectral invariants of ECH} \label{sub: c_sigma of ECH}
Let $(Y,\lambda)$ be an oriented closed contact 3-dimensional manifold with contact structure $\xi$ and contact form $\lambda$, which satisfies $\lambda\wedge d\lambda> 0$ everywhere on $Y$. Each contact form $\lambda$ gives rise to a unique Reeb vector field $R$, satisfying the conditions $\lambda(R)=1$ and $R\in\ker d\lambda$. Let $\varphi_{\lambda}^t$ be the corresponding flow of $R$. The periodic orbits of $R$ are referred to as Reeb orbits of $\lambda$, and we denote the collection of all Reeb orbits of $\lambda$ as $\mathcal P(Y,\lambda)$. Let $\CL(Y,\lambda)$ denote the collection of Reeb $1$-currents, which are finite $\mathbb{R}_+$-linear combinations of elements in $\mathcal{P}(Y,\lambda)$ with positive coefficients. Let $\CL_{\mathbb{Z}_+}(Y,\lambda)$ denote the Reeb $1$-currents with coefficients in $\mathbb{Z}_+$.  If $\gamma\in \mathcal P(Y,\lambda)$ is a Reeb orbit with period $T$, then the linearized Poincar\'e map of $\gamma$ is defined as linearized time $T$ map $P_\gamma=d\varphi^T_{\lambda}:\xi_{\gamma(0)}\rightarrow \xi_{\gamma(T)}$, which is symplectic with respect to $d\lambda$.

Let $c_1(\xi)$ be the first Chern class of $(\xi,J)$, where $J$ is an almost complex structure on $\xi$ satisfying $d\lambda_x(\cdot,J_x\cdot)>0,\forall x\in Y$. Let $\mathbb{Z}_d:=\mathbb{Z}/d\mathbb{Z}$. For any $\Gamma\in H_1(Y,\mathbb{Z})$, one can define a $\mathbb{Z}_2$-vector space $\mathrm{ECC}_*(Y,\lambda,\Gamma)$ generated by ECH generators with a relative $\mathbb{Z}_d$-grading, where an ECH generator $\alpha$ has the form $\alpha=\sum_{i=1}^km_i\gamma_i\in \mathcal L_{\mathbb{Z}_+}(Y,\lambda)$ with $[\alpha]=\Gamma$ and $m_i=1$ whenever $P_{\gamma_i}$ has eigenvalues in $\mathbb{R}\setminus\{\pm 1\}$ and $d$ denotes the divisibility of $c_1(\xi)+2\mathrm{PD}(\Gamma)\in H^2(Y,\mathbb{Z})$. In particular, there always exists a $\Gamma$ such that $c_1(\xi)+2\mathrm{PD}(\Gamma)\in H^2(Y,\mathbb{Z})$ is torsion. Fix such $\Gamma$, then $\mathrm{ECC}_*(Y,\lambda,\Gamma)$ has a relative $\mathbb{Z}$-grading and one can define the homology of $\mathrm{ECC}_*(Y,\lambda,\Gamma)$ with a differential, i.e., $\mathrm{ECH}_*(Y,\lambda,\Gamma)$. By Taubes \cite{Taubes2010}, the homology is independent of $\lambda$, then we can also write $\mathrm{ECH}_*(Y,\xi,\Gamma)$.


Let $\mathcal{A}:\mathcal L(Y,\lambda)\cap \Gamma\rightarrow \mathbb{R}_{\geq0} (\alpha=\sum_{i=1}^ka_i\gamma_i\mapsto \int_\alpha\lambda=\sum_{i=1}^ka_iT_i)$, where $T_i$ is the period of $\gamma_i$. Let $\mathrm{ECC}^a_*(Y,\lambda,\Gamma)$ be the subcomplex of $\mathrm{ECC}_*(Y,\lambda,\Gamma)$, which is generated by ECH generators $\alpha$ with $\mathcal{A}(\alpha)<a$,
then we can define the filtered ECH as the homology of this subcomplex, which we denote by $\mathrm{ECH}^a_*(Y,\lambda,\Gamma)$. Then for any element $\varsigma\in \mathrm{ECH}_*(Y,\lambda,\Gamma)\setminus\{0\}$, one can defined a spectral invariant
\begin{equation}\label{equ: spectral invariant}
c_\varsigma(Y,\lambda)=\inf\left\{a\in \mathbb{R}_+:\ \varsigma\in \mathrm{Im}\ i^a:\mathrm{ECH}^a_*(Y,\lambda,\Gamma)\rightarrow \mathrm{ECH}_*(Y,\xi,\Gamma)\right\}.
\end{equation}
Combining the results in \cite{C-GHR2015,Hutchings2011,Irie2015}, we have the following properties.
\begin{prop}\label{prop: properties of spectral invariant}
The spectral invariants satisfy the following properties:
\begin{itemize}
\item[(1)] If $c_\varsigma(Y,\lambda)>0$, then there exists $\alpha\in \mathcal L_{\mathbb{Z}_+}(Y,\lambda)$ such that $c_\varsigma(Y,\lambda)=\mathcal{A}(\alpha)$.
\item[(2)] $c_\varsigma(Y,a\lambda)=ac_\varsigma(Y,\lambda),\ \forall a>0$.
\item[(3)] $c_\varsigma(Y,\lambda)\leq c_\varsigma(Y,f\lambda)$ for every $f\in C^\infty(Y,\mathbb{R}_{\geq 1})$.
\item[(4)] $\displaystyle \lim_{i\rightarrow +\infty}c_\varsigma(Y,f_i\lambda)=c_\varsigma(Y,\lambda)$ if $\displaystyle\lim_{i\rightarrow +\infty} f_i=1$ in $C^0$-topology.
\item[(5)] There exists a sequence $\{\varsigma_k\}_{k\geq 1}\subset \mathrm{ECH}_*(Y,\xi,\Gamma)$ such that
\begin{eqnarray*}
\lim_{k\rightarrow +\infty}\frac{c_{\varsigma_k}(Y,\lambda)}{\sqrt{2k}}=\sqrt{\mathrm{vol(Y,\lambda)}},\quad
 \mathrm{vol(Y,\lambda)}:=\int_Y\lambda\wedge d\lambda.
\end{eqnarray*}
\end{itemize}
\end{prop}
As a consequence, Irie proved the following results in \cite{Irie2018}.
\begin{prop}[Lemma 3.7 \cite{Irie2018}]\label{prop: lipschitz property}
Let $\gamma_k(\lambda):=\frac{c_{\varsigma_k}(Y,\lambda)}{\sqrt{2k}}-\sqrt{\mathrm{vol}(Y,\lambda)},\ \forall k\geq 1$. Then there exists a number $c(\lambda)>0$ such that
$$
|\gamma_k(e^{f_1}\lambda)-\gamma_k(e^{f_2}\lambda)|\leq c(\lambda)\|f_1-f_2\|_{C^0},\quad \forall f_1,f_2\in C^{\infty}(Y,[-1,1]),\quad \forall k\geq 1.
$$
\end{prop}

\begin{prop}[Lemma 3.4 \cite{Irie2018}]\label{prop: regularity of spectral invariant}
Let $\{\lambda_{\tau}:\tau\in[0,1]^N\}$ denote a smooth $[0,1]^N$-family of contact forms in $(Y,\xi)$ and $\varsigma\in\mathrm{ECH}(Y,\xi,\Gamma)\setminus\{0\}$. If $\lambda_{\tau_0}$ is nondegenerate, then
\begin{itemize}
\item[(1)] There exists a $C\in \mathcal L_{\mathbb{Z}_+}(Y,\lambda_{\tau_0})$ such that $C(\lambda_{\tau_0})=c_\varsigma(Y,\lambda_{\tau_0})$.
\item[(2)] Assume $\tau\mapsto c_\varsigma(Y,\lambda_\tau)$ is differentiable at $\tau_0$, then
$$C\big(\partial_{\tau_i} \lambda_\tau\big|_{\tau_0}\big)=\frac{\partial c_{\varsigma}(Y,\lambda_\tau)}{\partial \tau_i}\bigg|_{\tau_0},\quad \tau:=(\tau_1,\cdots,\tau_N).$$
\end{itemize}
\end{prop}

\subsection{Finsler and Riemannian settings}\label{sub: Finsler and Rie setting}
Let $(\Sigma,F)$ be a closed Finsler manifold, where $F$ is the Finsler metric that satisfies
\begin{itemize}
\item[(I)] $F:T\Sigma\setminus \Sigma \rightarrow \mathbb{R}_+$ is a smooth function and $F|_{\Sigma}=0$.
\item[(II)] $F(q,av)=aF(x,v)$ for every $(q,v)\in T\Sigma\setminus\Sigma$ and $a>0$.
\item[(III)] $\nabla^2_v F^2(q,v)$ is positive definite for every $(q,v)\in T\Sigma\setminus\Sigma$.
\end{itemize}
We call $F$ symmetric (or reversible) if
\begin{itemize}
\item[(II')] $F(q,av)=|a|F(q,v)$ for every $(q,v)\in T\Sigma\setminus\Sigma$ and $a\neq0$.
\end{itemize}
Let $\mathcal{F}(\Sigma)$ and $\mathcal{F}^s(\Sigma)$ denote the space of all Finsler and symmetric Finsler metrics on $\Sigma$, respectively. By fixing an arbitrary $F\in\mathcal{F}(\Sigma)$, one can rewrite
$$\mathcal{F}(\Sigma)=\bigg\{fF\in C^\infty(S_F\Sigma,\mathbb{R}_{+}): f=F_1/F\in C^\infty(S_F\Sigma,\mathbb{R}_{+}),\ \forall F_1\in\mathcal{F}(\Sigma)\bigg\},$$ where $S_F\Sigma=\{(q,v)\in T\Sigma: F(v)=1\}$ is the unit tangent bundle over $(\Sigma,F)$ and $f(q,av)=f(q,v)$ for every $(q,v)\in S_F\Sigma$ and $a>0$. Similarly, by fixing an arbitrary $F\in\mathcal{F}^s(\Sigma)$, $\mathcal{F}^s(\Sigma)$ can be expressed as the set of $f=F_1/F\in C^\infty(S_F\Sigma,\mathbb{R}_+)$ with $f(q,av)=f(q,v)$ for every $F_1\in \mathcal{F}^s(\Sigma)$, $(q,v)\in S_F\Sigma$ and $a\in \mathbb R_*:=\mathbb{R}\setminus\{0\}$. The $C^\infty$-topology of $\mathcal{F}(\Sigma)$ and $\mathcal{F}^s(\Sigma)$ is induced by
\begin{equation}\label{C infty norm}
\|f_1-f_2\|_{C^\infty}:=\sum^{\infty}_{k=0}\frac{\|\nabla^k (f_1-f_2)\|_{C^0}}{1+\|\nabla^k(f_1-f_2)\|_{C^0}},\quad \forall f_1,f_2\in C^\infty(S_F\Sigma,\mathbb{R}).
\end{equation}
Since $F^2$ is strictly convex on $T\Sigma$, by Legendre transformation, there exists a $1$-to-$1$ correspondence
\begin{equation}\label{equ: lagrangian transformation}
\mathcal{L}_F: T\Sigma\rightarrow T^*\Sigma\quad \left(v= v_1 \partial_{q_1}+v_2 \partial_{q_2}\mapsto F(v)\big(\partial_{v_1} F(v)dq_1+\partial_{v_2} F(v)dq_2\big)\right),
\end{equation}
which is independent of the choice of local coordinates $(q_1,q_2)$. Since $\mathcal{L}_F$ satisfies $\mathcal{L}_F(av)=a\mathcal{L}_F(v)$ for every $a>0$ (or for every $a\in\mathbb{R}_*$ if $F$ is symmetric), as discussed in (\ref{equ: homogenuity of LF}), we can introduce a dual metric $F^*(p)=F(\mathcal{L}_F^{-1}p)$ on $T^*\Sigma$, which also satisfies conditions (I)$-$(III) and (II'). Let $S^*_F\Sigma=\{(q,p)\in T^*\Sigma: F^*(p)=1\}$ denote the unit cotangent bundle and let ${\lambda_F}|_e:=e\circ \pi_*,\forall e\in S^*_F\Sigma$ be the canonical Liouville form, where $\pi_*$ is differential of projection $\pi:T^*\Sigma\rightarrow \Sigma$. Condition (III) implies that $S^*_F\Sigma$ admits convex fibers or symmetric convex fibers if $F\in\mathcal{F}^s(\Sigma)$. The pair $(S^*_F\Sigma,\xi_F=\ker \lambda_F)$ constitutes a contact manifold and satisfies the isomorphism $(S^*_F\Sigma,(F_1^*)^{-1}\lambda_F)\cong (S^*_{F_1}\Sigma,\lambda_{F_1})$ for every $F_1\in\mathcal{F}(\Sigma)$, as discussed in section \ref{sub: geodesic and Reeb flows}. Let $\Lambda_F:=\{f\lambda_F: f\in C^\infty(S^*_{F}\Sigma,\mathbb R_*)\}$ denote the collection of all contact forms on $(S^*_F\Sigma,\xi_F)$. Define
\begin{equation}\label{equ: lambda Fin}
\Lambda^{Fin}_F:=\left\{\mathrm{contact\ forms}\ \lambda_{F_1}=(F_1^*)^{-1}\lambda_F:\forall F_1\in\mathcal{F}(\Sigma)\right\}\subset \Lambda_F.
\end{equation}
Similarly, we have $\Lambda^{sFin}_F$ for every $F_1\in\mathcal{F}^s(\Sigma)$ by fixing an $F\in\mathcal{F}^s(\Sigma)$ in advance. The $C^\infty$-topology of $\Lambda^{Fin}_F$ (or $\Lambda^{sFin}_F$) is induced by
$$\big\|\lambda_{F_1}-\lambda_{F_2}\big\|:=\big\|F^{-1}F_1-F^{-1}F_2\big\|_{C^\infty},\quad \forall F_1,F_2\in\mathcal{F}(\Sigma)\ (\text{or}\ \mathcal{F}^s(\Sigma)).$$
Then $\Lambda^{Fin}_F$ and $\mathcal{F}(\Sigma)$ (as well as $\Lambda^{sFin}_F$ and $\mathcal{F}^s(\Sigma)$) are homeomorphic to each other. Moreover, we observe from \eqref{equ: lambda Fin} that $\mathcal F(\Sigma)$ associates to an open subset of $\Lambda_F$. In another word, every perturbed contact form $f\lambda_F$ on $(S^*_F\Sigma,\xi_F)$ with $\|f-1\|_{C^\infty}$ sufficiently small leads to a perturbed Finsler metric $F_1$ satisfying $f^{-1}=F_1^*|_{S^*_F\Sigma}$ due to the fiber convexity of $F_1^*$.

As a special case, we let $\mathcal{M}(\Sigma)\subset \mathcal{F}^s(\Sigma)\ (g\hookrightarrow F_g:=\|\cdot\|_g)$ denote the space of Riemannian metrics on $\Sigma$. Then by fixing $g\in \mathcal{M}(\Sigma)$, one can rephrase
$$\mathcal{M}(\Sigma):=\left\{f_{g_1}^2g\in C^\infty(S_g\Sigma,\mathbb{R}_+): f_{g_1}(\cdot)=\|\cdot\|_g^{-1}\|\cdot\|_{g_1}\in C^\infty(S_g\Sigma,\mathbb{R}_+),\ \forall g_1\in\mathcal{M}(\Sigma)\right\},$$
where $S_g\Sigma:=S_{F_g}\Sigma$. The $C^\infty$-topology of $\mathcal{M}(\Sigma)$ is also induced by the norm defined in (\ref{C infty norm}). Based on the transformation (\ref{equ: lagrangian transformation}), we can deduce that $\mathcal{L}_{g}(v)=(v,\cdot)_g$, where $\mathcal{L}_g:=\mathcal{L}_{F_g}$. The dual metric $g^*$ on $T^*\Sigma$ is then defined as $\|\cdot\|_{g^*}:=F_g^*$, where the relation $\|\mathcal{L}_g(v)\|_{g^*}=\|v\|_g$ holds. Consequently, we obtain a contact manifold $(S^*_g\Sigma,\xi_g):=(S^*_{F_g}\Sigma,\xi_{F_g})$, where $\xi_g$ denotes the kernel of the Liouville form $\lambda_g:=\lambda_{F_g}$ on the $g^*$-unit cotangent bundle $S^*_g\Sigma$. Similar to $\Lambda^{Fin}_F$, we define
\begin{equation}\label{equ: lambda Rie}
\Lambda^{Rie}_g:=\left\{\mathrm{contact\ forms}\ \lambda_{g_1}=\|\cdot\|^{-1}_{g_1^*}\lambda_g,\ \forall g_1\in\mathcal{M}(\Sigma)\right\}.
\end{equation}
The $C^\infty$-topology of $\Lambda^{Rie}_g$ is induced by the norm $\|\lambda_{g_1}-\lambda_{g_2}\|:=\|f_{g_1}-f_{g_2}\|_{C^\infty},\forall g_1,g_2\in\mathcal{M}(\Sigma),$
where $f_{g_i}=\|\cdot\|_g^{-1}\|\cdot\|_{g_i}, i=1,2$. Therefore, $\Lambda^{Rie}_g$ and $\mathcal{M}(\Sigma)$ are homeomorphic to each other.

Let $\mathrm{CG}^{F}$ denote the set of all closed geodesics on $(\Sigma,F)$ with parametrization, and let $\mathrm{CG}^{g}:=\mathrm{CG}^{F_g}$. For each $\sigma\in\mathrm{CG}^{F}$, we define $\tilde\sigma=(\sigma,\dot\sigma)$ and $\gamma_\sigma=\mathcal{L}_F(\tilde\sigma)$ as the corresponding curve on $S_F\Sigma$ and $S^*_F\Sigma$, respectively. A Finsler $1$-current on $(\Sigma,F)$ is a formal linear combination $a_1\sigma_1+\cdots+a_k\sigma_k$, where $a_i>0$ and $\sigma_i\in \mathrm{CG}^{F}$ for every $i=1,\cdots,k$. Similarly, we have Riemannian $1$-currents on $(\Sigma,g)$. We denote $\CL^{F}$ and $\CL^g$ the set of all Finsler and Riemannian $1$-currents on $(\Sigma,F)$ and $(\Sigma,g)$, respectively. From another perspective, the Finsler and Reeb $1$-currents are linear functions on the spaces of differential $1$-forms, i.e.
\begin{equation}\label{equ: current acting on forms}
\sum_{i=1}^k a_i\sigma_i(\alpha):=\sum_{i=1}^k a_i\int_{\tilde \sigma_i}\alpha,\quad \sum_{i=1}^k a_i\gamma_{\sigma_i}(\beta):=\sum_{i=1}^k a_i\int_{\gamma_{\sigma_i}}\beta,\quad \forall \alpha\in \Omega^1(S_F\Sigma),\ \beta\in \Omega^1(S^*_F\Sigma).
\end{equation}

\subsection{Geodesic flows and Reeb flows}\label{sub: geodesic and Reeb flows}
Consider a Finsler surface $(\Sigma,F)$ with $F$ satisfying (I)$-$(III). Denote $F_q(v):=F(q,v)$. Since $F^2_q$ is a convex function on $T^*_q\Sigma$ for every $q\in \Sigma$, the Legendre transformation $\mathcal{L}_F$ defined in (\ref{equ: lagrangian transformation}) is a one-to-one correspondence from $T_q\Sigma$ to $T^*_q\Sigma$. Choose $q=(q_1,q_2)$ as the local coordinates of $\Sigma$. By condition (II), we have $aF_q(v)=F_q(av)$ for every $a>0$ and $v=\sum_{i=1}^2v_i\partial_{q_i}\in T_q\Sigma$. By differentiating $v_i$ and $a$ on both sides, we obtain that
\begin{equation}\label{equ: homogenuity of F}
\partial_{v_i} F_q(av)=\partial_{v_i} F_q(v),\ i=1,2,\quad  \text{and}\quad  F_q(v)=d(F_q)|_{av}(v),\quad \forall a>0,
\end{equation}
then we have
\begin{equation}\label{equ: homogenuity of LF}
\mathcal{L}_F(av)=F_q(av)\sum_{i=1}^2\partial_{v_i} F_q(av)dq_i=a\mathcal{L}_F(v),\quad \forall a>0.
\end{equation}
If $F\in\mathcal{F}^s(\Sigma)$, then (\ref{equ: homogenuity of LF}) holds for every $a\neq 0$. Recall that $(S^*_F\Sigma,\lambda_F)$ constitutes a contact manifold. Under the local coordinates $(q,p):=pdq$ of $T^*\Sigma$, the contact form $\lambda_F$ can be locally written as the restriction of the $1$-form $\lambda_0=pdq$ on $S^*_F\Sigma$.

To measure the length of closed geodesics, we introduce the Hilbert form as
\begin{equation}\label{equ: def of Hilbert form}
\beta=\partial_{v_1} Fdq_1+\partial_{v_2} Fdq_2,\quad \beta_F=\beta|_{S_F\Sigma},
\end{equation}
which is globally well-defined and independent of the choice of local coordinates $(q_1,q_2)$. According to \cite{DGZ2017}, the pair $(S_F\Sigma,\beta_F)$ constitutes a contact manifold and the Reeb flow of $\beta_F$ coincide with the geodesic flow on $S^*_F\Sigma$. For every $\sigma\in \mathrm{CG}^F$ with $\tilde\sigma=(\sigma,\dot\sigma)$ in $T\Sigma$, we have
\begin{equation}\label{equ: length function}
l_F(\sigma)=\int_0^{l_F(\sigma)} F_{\sigma(t)}(\dot\sigma(t))dt\overset{\eqref{equ: homogenuity of F}}{=}\int_0^{l_F(\sigma)}d(F_{\sigma(t)})|_{\dot\sigma(t)}(\dot\sigma(t))dt=\int_{\tilde\sigma}\beta_F.
\end{equation}

To measure the volume of $(\Sigma,F)$, we define the volume form of $(\Sigma,F)$ as
\begin{equation}\label{equ: volume form}
d\mathrm{vol}_F:=-\frac{1}{\pi}\left(\int_{D^*_{F,q}\Sigma}dp_1\wedge dp_2\right)dq_1\wedge dq_2,
\end{equation}
where $D^*_{F,q}\Sigma$ is the unit disk bounded by the fiber $S^*_{F,q}\Sigma\subset S^*_F\Sigma$ over $q$. The minus sign is needed to adapt the orientation of $S^*_F\Sigma$ induced by $\lambda_F\wedge d\lambda_F$. Note that $d\mathrm{vol}_F$ is globally well-defined and the sign corrects the orientation. Moreover, we define $d\mathrm{vol}_F(f):=\int_\Sigma f(q) d\mathrm{vol}_F,\forall f\in C^\infty(\Sigma,\mathbb{R})$. In particular, the volume of $\Sigma$ is $\mathrm{vol}(\Sigma,F)=d\mathrm{vol}_F(1)$.

If $F=F_g$ satisfies $F_g^2=g\in\mathcal M(\Sigma)$, we recall that $\|\cdot\|_{g^*}=\|\mathcal{L}_g^{-1}(\cdot)\|_g$, where $\mathcal{L}_g(v)=\mathcal{L}_{F_g}(v)=(v,\cdot)_g,\forall v\in T_q\Sigma$. Under the local coordinates $(q_1,q_2)$, we have $p:=\mathcal{L}_{g}(v)=p_1dq_1+p_2dq_2$, where $p_i=g_{i1}v_1+g_{i2}v_2$ and  $g_{ij}=(\partial_{q_i},\partial_{q_j} )_g$. The Hilbert form becomes $\beta_g|_{(q,v)}:=\beta_{F_g}|_{(q,v)}=(v,\cdot)_g$ for every $(q,v)\in S_g\Sigma$. Let $G=(g_{ij})_{2\times 2}$, $G^*=(g_{ij}^*)_{2\times 2},\ \bar v=(v_1,v_2)^T,\ \bar p=(p_1,p_2)^T=G\bar v$. We have
$$(G\bar v,\bar v)=(v,v)_g=(p,p)_{g^*}=(G^*\bar p,\bar p)=(G^*G\bar v,G\bar v),$$ then $G^*=G^{-1}$. Since
\begin{equation}\label{equ: area of fiber}
\int_{D^*_{g,q}\Sigma}dp_1\wedge dp_2=\mathrm{Area}\{\bar p\in\mathbb{R}^2|\ (G^*\bar p,\bar p)=1\}=\pi\det(G)^{\frac{1}{2}},
\end{equation}
then $d\mathrm{vol}_{F_g}=-\det(G)^{\frac{1}{2}}dq_1\wedge dq_2$ coincides with the canonical volume form up to a sign. For the convenience to the proofs in later sections, we list the following useful relations between the geodesic flows and Reeb flows.

\begin{prop}\label{prop: properties of Finsler geodesic and Reeb flows}
For any $F,F_1\in\mathcal{F}(\Sigma)$, we have
\begin{itemize}
\item[(1)] $(S^*_{F}\Sigma,(F_1^*)^{-1}\lambda_F)\cong(S^*_{F_1}\Sigma, \lambda_{F_1})\ \left((q,p)\overset{\phi}{\mapsto} (q,{F^*_1(p)}^{-1}p)\right)$.
\item[(2)] For any $f\in C^\infty(\Sigma,\mathbb{R}_+)$, we have
$$ (S^*_F\Sigma, (f\circ \pi)\lambda_F)\cong(S^*_{fF}\Sigma,\lambda_{fF})\ \left((q,p)\overset{\phi}{\mapsto} (q,f(q)p)\right).$$
\item[(3)] The Hamiltonian flow $\varphi^t_{F^*}$ of $F^*$ coincides with the Reeb flow $\varphi^{t}_{\lambda_F}$ on $S^*_F\Sigma$ and conjugates to the geodesic flow $\varphi_F^t$ on $S_F\Sigma$ via $\mathcal{L}_F$.
\item[(4)] The projection $\mathcal P(S^*_F\Sigma,\lambda_F)\rightarrow \mathrm{CG}^F\ (\gamma_\sigma\overset{\pi}{\mapsto} \sigma)$ is a one-to-one correspondence.
\item[(5)] $\gamma_\sigma(f\lambda_F)=\sigma(f\beta_F)$ for every $\gamma\in \mathrm{CG}^F$ and $f\in C^\infty(\Sigma,\mathbb{R})$.
\item[(6)] The period $T$ of $\gamma_\sigma\in \mathcal P(S^*_F\Sigma,\lambda_F)$ coincides with the length $l_F(\sigma)$.
\item[(7)] $\int_{S^*_F\Sigma} f\lambda_F\wedge d\lambda_F=2\pi d\mathrm{vol}_F(f)$ for every $f\in C^\infty(\Sigma,\mathbb{R})$.
\item[(8)] The volume satisfies $\mathrm{vol}(S^*_F\Sigma,\lambda_F)=2\pi\mathrm{vol}_F\Sigma$.
\end{itemize}
\end{prop}
\begin{proof}
To show (1)-(2). Since $\lambda_{F_1}=pdq|_{F_1^*(p)=1}$, (1) directly follows from the follows
$$\phi^*\lambda_{F_1}=F_1^*(p)^{-1}pdq|_{F^*(p)=1}=F_1^*(p)^{-1}\lambda_{F}.$$
Similarly, we have $\phi^*\lambda_{fF}=f\circ \lambda_F$. Notice that $(fF)^*=F^*/f$.

To show (3)-(4). Let $X_{F^*}$ be the Hamiltonian vector field of $F^*$ on $T^*\Sigma$. Denote $F^*_q$ the restriction of $F^*$ on $T^*_{q}\Sigma$. Since $F^*_q(ap)=aF^*_q(p),\forall a>0$, we have $d(F^*_q)|_{ap}(p\partial_p)=F^*_q(p)=1$ for every $a>0$ and $p\in S^*_{F,q}\Sigma$. Then
$\lambda_F(X_{F^*})=\omega_0(p\partial_p,X_{F^*})=d(F^*_p)|_p(p\partial_p)=1$ and $\omega_0(\cdot, X_{F^*})=dF^*=0$ on $S^*_F\Sigma$, which means that $X_{F^*}$ is the Reeb vector field of $\lambda_F$. Moreover, using (\ref{equ: lagrangian transformation}), we compute
$$
\mathcal{L}_F^*\lambda_F=F(v)\sum_{i=1}^2\partial_{v_i} F(v)dq_i=\sum_{i=1}^2\partial_{v_i} F(v)dq_i=\beta_F,\quad \forall (q,v)\in S_F\Sigma.
$$
Then $d\mathcal{L}_F$ sends the Reeb vector field $\beta_F$ to the Reeb vector field of $\lambda_F$. Since the Reeb flow of $\beta_F$ coincides with the geodesic flow on $S_F\Sigma$, see \cite{DGZ2017}, then (3) follows. (4) follows from (3), directly.

To show (5)-(6). Since $\gamma_\sigma=\mathcal{L}_F(\tilde \sigma)$, by definition (\ref{equ: current acting on forms}), we have
\begin{eqnarray*}
\gamma_\sigma(f\lambda_F)=\int_{\gamma_\sigma}f\lambda_F=\int_{\tilde\sigma} (f\circ \mathcal{L}_F)\mathcal{L}_F^*\lambda_F=\int_{\tilde\sigma} f\beta_F=\sigma(f\beta_F).
\end{eqnarray*}
When $f\equiv 1$, we have (6), i.e. $T=l_F(\sigma)$.

To show (7)-(8). Using definition (\ref{equ: volume form}), we compute
\begin{equation*}
\begin{aligned}
\int_{S^*_F\Sigma} f\lambda_F\wedge d\lambda_F&=\int_{\Sigma} f(q)\left(\int_{S^*_{F,q}\Sigma}p_2dp_1-p_1dp_2\right)dq_1\wedge dq_2\\
&=2\int_{\Sigma} f(q)\left(-\int_{D^*_{F,q}\Sigma}dp_1\wedge dp_2\right)dq_1\wedge dq_2=2\pi\int_\Sigma f(q)d\mathrm{vol}_F\\
&=2\pi d\mathrm{vol}_F(f).
\end{aligned}
\end{equation*}
(8) is the conclusion of (7), when $f\equiv 1$.
\end{proof}
As a consequence, we can obtain a similar result for Riemannian case.
\begin{prop}\label{prop: properties of Riemannian geodesic and Reeb flows}
For any $g,g_1\in\mathcal{M}(\Sigma)$, we have
\begin{itemize}
\item[(1)] $(S^*_{g}\Sigma,\|p\|^{-1}_{g_1^*}\lambda_g)\cong(S^*_{g_1}\Sigma, \lambda_{g_1})\ \left((q,p)\mapsto (q,\|p\|^{-1}_{g_1^*}p)\right)$.
\item[(2)] For any $f\in C^\infty(\Sigma,\mathbb{R}_+)$, we have
$$ (S^*_g\Sigma, (f\circ \pi)\lambda_g)\cong(S^*_{f^2g}\Sigma,\lambda_{f^2g})\ \big((q,p)\mapsto (q,f(q)p)\big).$$
\item[(3)] The Hamiltonian flow $\varphi^t_{g^*}$ of $\|\cdot\|_{g^*}$ coincides with the Reeb flow $\varphi^{t}_{\lambda_g}$ and conjugate to the geodesic flow $\varphi_g^t$ on $S_g\Sigma$ via $\mathcal{L}_g$.
\item[(4)] The projection $\mathcal P(S^*_g\Sigma,\lambda_g)\rightarrow \mathrm{CG}^g\ (\gamma_\sigma\overset{\pi}{\mapsto} \sigma)$ is a one-to-one correspondence.
\item[(5)] $\gamma_\sigma(f\lambda_g)=\sigma(fds)$ for every $\sigma\in \mathrm{CG}^g$ and $f\in C^\infty(\Sigma,\mathbb{R})$.
\item[(6)] The period $T$ of $\gamma_\sigma\in \mathcal P(S^*_g\Sigma,\lambda_g)$ coincides with the length $l_g(\sigma)$.
\item[(7)] $\int_{S^*_g\Sigma} f\lambda_g\wedge d\lambda_g=2\pi d\mathrm{vol}_g(f)$ for every $f\in C^\infty(\Sigma,\mathbb{R})$.
\item[(8)] The volume $\mathrm{vol}(S^*_g\Sigma,\lambda_g)=2\pi\mathrm{vol}_g\Sigma$.
\end{itemize}
\end{prop}

Therefore, the mapping $\mathcal{P}(S^*_F\Sigma,\lambda_F)\rightarrow \mathrm{CG}^F (\gamma_{\sigma}\mapsto \sigma)$ establishes a one-to-one correspondence. Moreover, the generalized mapping from $\CL(S^*_F\Sigma,\lambda_F)$ to $\CL^{F}$ also establishes a one-to-one correspondence. In particular, there are one-to-one correspondences from $\CL(S^*_F\Sigma,\lambda_F)$ to $\CL^F$ for every $F\in \mathcal F^s(\Sigma)$ and from $\CL(S^*_g\Sigma,\lambda_g)$ to $\CL^g$ for every $g\in \mathcal M(\Sigma)$, due to the relations $\mathcal{M}(\Sigma)\subset \mathcal{F}^s(\Sigma)\subset \mathcal{F}(\Sigma)$ and $\Lambda^{Rie}_g\subset \Lambda^{sFin}_{F_g}\subset \Lambda^{Fin}_{F_g}\subset \Lambda_F$.

On a contact manifold $(Y,\lambda)$, we call a Reeb orbit $\gamma\in \mathcal P(Y,\lambda)$ nondegenerate if $P_{\gamma}$ does not have eigenvalue $1$. The contact form $\lambda$ is nondegenerate if all the Reeb orbits are nondegenerate. By Proposition \ref{prop: properties of Finsler geodesic and Reeb flows}-(3), each closed geodesic $\sigma\in\mathrm{CG}^F$ associates to a Reeb orbit
\begin{equation}\label{equ: gamma of sigma}
\gamma_\sigma(t)=(\sigma(t),p(t))\in S^*_F\Sigma,\quad p(t)=\mathcal{L}_F(\dot\sigma(t)),\quad \forall t\in\mathbb{R}.
\end{equation}
We call the closed geodesic $\sigma$ nondegenerate if $\gamma_\sigma$ is nondegenerate. We call $F\in\mathcal{F}(\Sigma)$ bumpy if all closed geodesics are nondegenerate. A Finsler $1$-current $\sum_{i=1}^ka_i\sigma_i\in \mathcal{L}^F$ is considered nondegenerate if $\sigma_1,\cdots,\sigma_k$ are all nondegenerate. Analogous to the Reeb $1$-current in $\mathcal L(S^*_F\Sigma,\lambda_F)$. Therefore, $F$ is bumpy if and only if $\lambda_F$ is nondegenerate.

\subsection{Sard-Smale theorem}\label{sec: Sard-smale theorem}
In this section, we briefly introduce the Sard theorem and Sard-Smale theorem, which is mainly used in section~\ref{sec: generic nondegenercy of Rimannian metrics}.
A Fredholm operator is a continuous linear map $L:E_1\rightarrow E_2$ between Banach spaces, which satisfies $\dim\ker L<+\infty,\ \dim\mathrm{coker} L<\infty$ and admits a closed image. The Fredholm index of $L$ is defined as $\mathrm{Ind}L:=\dim\ker L-\dim\mathrm{coker}L\in\mathbb{Z}.$
A Fredholm map is a $C^1$-map $f:X\rightarrow Y$ between connected Banach manifolds such that
$Df(x): T_xX\rightarrow T_{f(x)}Y$ is Fredholm for every $x\in X$. The Fredholm index of $f$ is defined as $\mathrm{Ind}f:=\mathrm{Ind}Df(x)$ for some $x$, which is independent of $x$. A point $x\in X$ is called singular point of $f$ if $Df(x)$ is not surjective and $f(x)$ is called a singular value of $f$. $x\in X$ is called regular point if $x$ is not singular and $y\in Y$ is called a regular value of $f$ if $y$ is not a singular value. We firstly introduce the famous Sard theorem, see \cite{Sard1942}, \cite{Smale1965}.
\begin{lem}[Theorem 1.2 in \cite{Smale1965}]\label{lem: Appendix Sard theorem}
Let $M,N$ be $m,n$-dimensional manifold and $U\subset M$ is an open submanifold. If $f:U\rightarrow N$ is a $C^r$-map, where $r>\max\{0,m-n\}$, then the set of critical values of $f$ has measure zero in $N$.
\end{lem}
The Sard theorem was generalized to infinite dimensional case by Smale in \cite{Smale1965} (c.f. \cite{Chang2005}).
\begin{lem}[Theorem 1.3.16 in \cite{Chang2005}]\label{lem: Appendix Sard-Smale theorem}
Suppose $X$ is a separable Banach space, $U\subset X$ is an open set and $Y$ is a Banach space. Let $f: U\rightarrow Y$ be a Fredholm $C^r$-map, where $r>\max\{0,\mathrm{Ind}f\}$, then the regular set of $f$ is residual.
\end{lem}
There is also an infinite dimensional version of implicit function theorem as follows:
\begin{lem}[Theorem A.3.3 in \cite{MS2004}]\label{lem: Appendix IFT}
Let $X,Y$ be Banach spaces and let $U\subset X$ be an open subset. If $f: U\rightarrow Y$ is a $C^r$-map and $y$ is a regular value of $f$, then $f^{-1}(y)\subset X$ is a $C^r$ submanifold. Furthermore, if $f$ is Fredholm, then $\dim f^{-1}(y)=\mathrm{Ind}f$.
\end{lem}
Let $f:X\rightarrow Y$ be a $C^1$ map and let $W\subset Y$ be a submanifold. We call $f$ is transverse to $W$ if
$\mathrm{Im}Df(x)+T_{f(x)}W=T_{f(x)}Y,\forall x\in f^{-1}(W).$ As a corollary of Lemmas \ref{lem: Appendix Sard-Smale theorem} and \ref{lem: Appendix IFT}, we have
\begin{lem}[Theorem 3.3 in \cite{Smale1965}]\label{lem: Appendix Sard-Smale theorem 2}
Let $f: X\rightarrow Y$ be a $C^r$-map. Assume $f$ is transverse to a finitely dimensional smooth submanifold $W\subset Y$, then $f^{-1}(W)$ is a $C^{r}$-submanifold of $X$. Furthermore, if $f$ is Fredholm, then $\dim f^{-1}(W)=\mathrm{Ind}f+\dim W.$
\end{lem}

\section{Local variations and generic perturbations} \label{sec: generic nondegenercy of Rimannian metrics}
In this section, we aim to establish two results, namely Lemma \ref{lem: genericity of nondegenercy 1} and Theorem \ref{lem: genericity of nondegenercy 2}. 
We prove in Lemma \ref{lem: genericity of nondegenercy 1} that any Finsler $1$-current on $(\Sigma,F)$ can exhibit non-degeneracy under an infinitesimal perturbation of $F$ within $\mathcal{F}(\Sigma)$. Moreover, these perturbations preserve the symmetric Finsler and Riemannian structure if $F\in \mathcal{F}^s(\Sigma)$ and $\mathcal{M}(\Sigma)$. Theorem \ref{lem: genericity of nondegenercy 2} establishes that for any given $N\gg 0$, the bumpy metrics occupy a full-measure subset in $[0,1]^N$ for a generic set of smooth $[0,1]^N$-family of metrics in $\mathcal{F}(\Sigma)$, $\mathcal{F}^s(\Sigma)$ and $\mathcal{M}(\Sigma)$.

We shall first investigate the perturbation of metrics on the linearized Poincar\'e map $P_\sigma$. Consider $F,F_1\in\mathcal{F}(\Sigma)$ and  $f=-\ln F_{1}^*$. From Proposition~\ref{prop: properties of Finsler geodesic and Reeb flows}-(1), we observe that the contact form $e^{f}\lambda_F\in \Lambda^{Fin}_F$ on $S^*_F\Sigma$ associates to $F_1$. Let $U$ be a sufficiently small tubular neighborhood of $\gamma_\sigma\in \mathcal P(S^*_F\Sigma,\lambda_F)$ with $\sigma\in \mathrm{CG}^F$. If $f|_{\gamma_\sigma}\equiv0$ and $df|_{\gamma_\sigma}\equiv 0$, then $\gamma_\sigma$ belongs to $\mathcal P(S^*_F\Sigma,e^f\lambda_F)$, since the Reeb vector field does not change along $\gamma_\sigma$. Consider an open subset $\mathcal{U}\subset \Sigma$ equipped with local coordinates $(q,p)=pdq$ in $T^*\mathcal{U}$. Let $S^*_F\mathcal{U}:=S^*_F\Sigma\cap T^*\mathcal{U}$. We define
\begin{equation}\label{equ: collection of perturbation function f}
\begin{aligned}
\mathcal{F}^{pert}_{\mathcal{U}}:&=\left \{f\in C^\infty_0(S^*_F\mathcal{U},\mathbb{R}):\ f=-\frac{1}{2}\ln\left(1+(A(q)p,p)\right)\right\},\\
\mathcal{F}^{pert}_{\mathcal{U},0}:&=\left \{f\in \mathcal{F}^{pert}_{\mathcal{U}},\ f\big|_{\gamma_\sigma}\equiv 0,\ df\big|_{\gamma_\sigma}\equiv 0\right\},
\end{aligned}
\end{equation}
where $A(q)$ is a symmetric $2\times 2$ matrix that smoothly depends on $q$ with support lies in $\mathcal{U}$. Moreover, if $\mathcal{U}=\pi(U)$, then $\mathcal{U}$ is a sufficiently small open neighborhood of $\sigma$ with local coordinates $q$ generated by the tangent vector field $\dot \sigma$ and the normal vector field along $\sigma$. However, the local coordinates are inadequate for extension to $\pi(U)$, as self-intersections of $\sigma$ might occur. Therefore, an additional requirement is necessary: $A(\sigma(t_1))=A(\sigma(t_2))$ for every $t_1\neq t_2$ satisfying $\sigma(t_1)=\sigma(t_2)$.


Consider a variation $F_s\in\mathcal{F}(\Sigma)$ of $F=F_0$ as follows. Denote $F^*_q(p):=F^*(q,p)$. Let $f_s\in \mathcal{F}^{pert}_{\pi(U),0}$. It can be extended to
$f_s(q,p)=-\frac{1}{2}\ln\left(F^*_q(p)^2+s(A(q)p,p)\right),\forall (q,p)\in T^*\pi(U)$,
Define the variation $\{F_s\}_{s\in(-\epsilon,\epsilon)}$ as
\begin{equation}\label{equ: variation Fs}
F_s:=F^*_s\circ \mathcal{L}_{F_s}\quad\text{and}\quad F_s^*:=e^{-f_s}=\left\{\begin{aligned}
&\big(F^*_q(p)^2+s(A(q)p,p)\big)^{1/2},\quad \forall (q,p)\in T^*\pi(U),\\
&F^*,\quad \text{elsewhere}.
\end{aligned}\right.
\end{equation}
Especially, the variation $F_s\in \mathcal{F}^s(\Sigma)$ or $\mathcal{M}(\Sigma)$, if $F\in \mathcal{F}^s(\Sigma)$ or $\mathcal{M}(\Sigma)$.

Now we represent $A(\cdot)$ in a small subset $U_1\subset U$, properly. Let $V=V_1\partial_{p_1}+V_2\partial_{p_2}$ be a nowhere vanishing section in $TU_1$, which tangents to the fibers of $S^*_{F}\Sigma$. By identifying $T_p(T^*_q\Sigma)$ and $T_q^*\Sigma$, we obtain a basis $\{\hat V(q,p),pdq\}$ in $T^*_{q}\Sigma$, where $\hat V(q,p)=V_1(q,p)dq_1+V_2(q,p)dq_2$. Under this basis, we obtain a matrization $A=\big(A_{ij}\big)_{2\times 2}$ of $A(q)$ depends on both $q$ and $p$. To remove the dependency in $p$, we do the following modifications. After a suitable translation, we assume that $l_\sigma/2\in \mathrm{reg}(\sigma)$, where $l_\sigma=l_F(\sigma)$ and $\mathrm{reg}(\sigma)\subset \mathbb{R}/(l_\sigma \mathbb{Z})$ denotes the complement of self-intersecting moments. Choose $U_1$ as a sufficiently small open neighborhood of $\gamma_\sigma(l_\sigma /2)$ so that there is no self-intersection points in the interior and $U_1\cap \sigma=\sigma|_{(l_\sigma/2-\epsilon,l_\sigma/2+\epsilon)}$. Then we choose a family of curves $\{\sigma_s\}_{s\in(-\epsilon,\epsilon)}$, which encompasses the projected domain $\pi(U_1)$ so that $\sigma_0=\sigma|_{(l_\sigma/2-\epsilon,l_\sigma/2+\epsilon)}$, where $\partial_s \sigma_s(t)\big|_{s=0}$ is a transverse vector field along $\sigma$ in $\pi(U_1)$. Note that
$$
\xi_F|_{U_1}=\ker pdq|_{U_1}=\{V,V'\}|_{U_1},\quad V'|_{U_1}=-p_2\partial_{q_1}+p_1\partial_{q_2},
$$
where $V'|_{\gamma_\sigma(t)}=-p_2(t)\partial_{q_1}+p_1(t)\partial_{q_2}$ transverse to $\dot \sigma$ along $\sigma$. Under a suitable re-scaling of $V$ and reparametrization of $\sigma_{s}(t)$, we may assume that $\{V,V'\}$ is a symplectic basis with respect to $d\lambda_F$ and $V'=\partial_s \sigma_s(t)$ after projecting to $\pi(U_1)$. Then we acquire new coordinates $(t_1,t_2)=\sigma_{t_1}(t_2)$ in $\pi(U_1)$, along with a new basis $\{\hat V_{t_1,t_2},P_{t_1,t_2}\}$ for each $T^*_{\sigma_{t_1}(t_2)}\Sigma$, where
\begin{equation}\label{equ: P and V}
P_{t_1,t_2}=\mathcal{L}_{F}\left(\frac{\dot\sigma_{t_1}(t_2)}{F(\dot\sigma_{t_1}(t_2))}\right)\in S^*_{F,\sigma_{t_1}(t_2)}\Sigma,\quad \hat V_{t_1,t_2}=\hat V(\sigma_{t_1}(t_2),P_{t_1,t_2})\in T^*_{\sigma_{t_1}(t_2)}\Sigma.
\end{equation}
Especially, $P_{0,t}=p(t)$ and $\hat V_{0,t}=\hat V_{\sigma(t),p(t)}$. Under this basis, we denote the matrization of $A(\cdot)$ as
\begin{equation}\label{equ: local representation of tilde A}
A(\sigma_{t_1}(t_2))=\big(\tilde A_{ij}(t_1,t_2)\big)_{2\times 2},\quad \forall q\in \pi(U_1),
\end{equation}
where $\tilde A_{11}(t_1,t_2)=(A(\sigma_{t_1}(t_2))\hat V_{t_1,t_2},\hat V_{t_1,t_2})$, $\tilde A_{12}(t_1,t_2)=(A(\sigma_{t_1}(t_2))P_{t_1,t_2},\hat V_{t_1,t_2})$ and $\tilde A_{22}(t_1,t_2)=(A(\sigma_{t_1}(t_2))P_{t_1,t_2},P_{t_1,t_2})$. Furthermore, for each $(t_1,t_2)$, we define another frame
\begin{equation}\label{equ: basis bar V and bar W}
\{V_{t_1,t_2}, W_{t_1,t_2}\}:=\left\{V(\sigma_{t_1}(t_2), P_{t_1,t}),\partial_{t_1}(\sigma_{t_1}(t_2), P_{t_1,t})\right\}\subset T_p\left(T^*_{\sigma_{t_1}(t_2)}\Sigma\right).
\end{equation}
Since $\omega_0(V,V')=1$ on $\sigma(t)$ and $W_{t_1,t_2}=V'(\sigma_{t_1}(t_2), P_{t_1,t})+a(t_1,t_2)V_{t_1,t_2}$ for some function $a(t_1,t_2)$, then $\{V_{t_1,t_2}, W_{t_1,t_2}\}$ is a new symplectic basis of $\xi_{F}$. Define
$h(q,p):=\partial_s f_s(q,p)\big|_{s=0}=-\frac{1}{2}(\tilde A(q)p,p)$. Under the basis $\{V_{0,t}, W_{0,t}\}$, the gradient of $h|_{\xi_F}$ along $\gamma_\sigma(t)$ becomes
\begin{equation}\label{equ: gradient of h}
\begin{aligned}
\nabla h\big|_{\xi_F}(\gamma_\sigma(t))&=\left(-(A(\sigma(t))p(t),\hat V_{0,t}),-\frac{1}{2}\partial_{t_1} ( A(\sigma_{t_1}(t))P_{t_1,t}, P_{t_1,t})\big|_{t_1=0}\right)\\
&=\left(-\tilde A_{12}(0,t),-\frac{1}{2}\partial_{t_1}\tilde A_{22}(0,t)\right),
\end{aligned}
\end{equation}
and the Hessian of $h|_{\xi_F}$ is
\begin{equation}\label{equ: hessen of h}
\displaystyle \nabla^2h\big|_{\xi_F}(\gamma_\sigma(t))=\begin{pmatrix}-\tilde A_{11}(0,t) & -\partial_{t_1}\tilde A_{12}(0,t)\\
-\partial_{t_1}\tilde A_{12}(0,t) & -\frac{1}{2}\partial^2_{t_1}\tilde A_{22}(0,t)\end{pmatrix},
\end{equation}
where $\{\tilde A_{ij}\}$ comes from \eqref{equ: local representation of tilde A}.
\begin{prop}\label{prop: condition of perturbation metrix tilde A}
Given $\sigma\in \mathrm{CG}^F$, a small open subset $U_1\subset U$ as above and a family of functions $\{f_s: s\in(-\epsilon,\epsilon)\}$ in $\mathcal{F}^{pert}_{\pi(U_1)}$. Then $f_s|_{\gamma_\sigma}\equiv 0$ and $df_s|_{\gamma_\sigma}\equiv 0$ if and only if
\begin{eqnarray*}
\tilde A_{22}(0,t)=\partial_{t_1}\tilde A_{22}(0,t)=\tilde A_{12}(0,t)=0,\quad \forall t\in\mathbb{R}/(l_\sigma\mathbb{Z}).
\end{eqnarray*}
\end{prop}
\begin{proof}
Recall $f_s(q,p)=-\frac{1}{2}\ln(1+(A(q)p,p))$ and $\gamma_\sigma(t)=(\sigma(t),p(t))$. According to \eqref{equ: P and V}, we have $\tilde A_{12}(0,t)=(A(\sigma(t))p(t),\hat V_{0,t})$ and $\tilde A_{22}(0,t)=(A(\sigma(t))p(t),p(t))$, where $\hat V_{0,t}$ associates to $V_{0,t}$, which tangents to the fiber $S^*_{F,\sigma(t)}\Sigma$ at $\gamma_\sigma(t)$. Then $f_s|_{\gamma_\sigma}\equiv 0$ is equivalent to $\tilde A_{22}(0,t)\equiv 0,\forall t$. Moreover, since $T_{p(t)}(S^*_F\Sigma)$ consists of three independent directions: $\partial_{t_1},\partial_{t}\in T_{\sigma(t)}\Sigma$ and $V_{0,t}$. We observe that $\tilde A_{22}(0,t)\equiv 0$ is equivalent to $\partial_{t}f_s|_{\gamma_\sigma}\equiv 0$ and $\tilde A_{12}(\sigma(t))\equiv 0$ is equivalent to $df_s(V_{0,t})|_{\gamma_\sigma}\equiv 0$. Moreover,
$$\begin{aligned}
\partial_{t_1}f_s|_{\gamma_\sigma(t)}&=\partial_{t_1}(A(\sigma_{t_1}(t))p(t),p(t))|_{t_1=0}\\
&=\partial_{t_1}(A(\sigma_{t_1}(t))P_{t_1,t},P_{t_1,t})_{t_1=0}-2(A(\sigma(t))p(t),a(t)\hat V_{0,t})\\
&=\partial_{t_1}\tilde A_{22}(0,t)-2a(0,t)\tilde A_{12}(0,t).
\end{aligned}$$
Here $P_{0,t}=p(t)$ and $\partial_{t_1}P_{0,t}=W_{0,t}-V'(\sigma(t), p(t))=a(0,t)V_{0,t}$ are used. Hence, this proposition follows.
\end{proof}

Now we consider the variation of the linearized Poincar\'e map $P_{\sigma}$ of $\sigma$ with respect to $e^{f_s}\lambda_F$ for every $f_s\in \mathcal{F}^{pert}_{\pi(U),0}$. Recall the linearized Poincar\'e return map $P_{\sigma}:=P_{\gamma_\sigma}:\mathcal{F}^{pert}_{\pi(U),0}\rightarrow \mathrm{Sp}(\xi_F|_{\gamma_\sigma(0)},\omega_0)$ sending $f$ to a symplectic linear map $P_{\sigma}(f)$, where $\mathrm{Sp}(\xi_F|_{\gamma_\sigma(0)},\omega_0)$ denotes the symplectic group on $(\xi_F|_{\gamma_\sigma(0)},\omega_0)$. Let $D_r^2\subset \mathbb C$ denote the standard disk with radius $r>0$.
Consider the Martinet tube $\Phi: \mathbb R/\mathbb Z\times D^2_\epsilon\rightarrow U$ around $\gamma_\sigma$ for some $\epsilon>0$ sufficiently small, that is $\Phi(\gamma(t))=(t/l_\sigma,0,0)$ and $\Phi^*\lambda_F=H(dz+x_1dx_2)$, where $H$ is a smooth positive function on $\mathbb R/\mathbb Z\times D^2_\epsilon$ satisfying $H(z,0,0)=l_\sigma$ and $dH(z,0,0)=0$. Using the Darboux theorem, we may further require that $H|_{\Phi(U_1)}\equiv l_\sigma$. Then we establish an identification between $U$ and $\mathbb R/\mathbb Z\times D^2_\epsilon$, along with the corresponding Reeb flows. We have the following result.

\begin{lem}\label{lem: differential of P_gamma(f)}
Given $\sigma\in \mathrm{CG}^F,U$ and $U_1$ as above. Consider a variation $\{f_s: s\in(-\epsilon,\epsilon)\}\subset \mathcal{F}^{pert}_{\pi(U),0}$ of $f_0=0$ and its initial derivative $h(q,p)=-\frac{1}{2}(A(q)p,p)$ for some $\epsilon>0$ sufficiently small, where the support of $A(\cdot)$ lies in $\pi(U_1)$ and satisfies $\mathrm{supp}A\cap \sigma= \sigma(l_\sigma/2-\epsilon,l_\sigma/2+\epsilon)$. In the Martinet tube $\Phi: \mathbb R/\mathbb Z\times D^2_\epsilon\rightarrow U$, we have
\begin{eqnarray*}
dP_{\sigma}(h)\big|_{f=0}=P_{\sigma}(0)R_0(\frac{l_\sigma}{2})^{-1}\left(\frac{1}{2l_\sigma}\int^{l_\sigma/2+\epsilon}_{l_\sigma/2-\epsilon}J\nabla^2_x\tilde h(\frac{t}{l_\sigma},0,0)dt\right)R_0(\frac{l_\sigma}{2}),\quad J=\begin{pmatrix} 0& -1\\ 1 & 0\end{pmatrix},
\end{eqnarray*}
where $\tilde h=h\circ \Phi$ and $R_0(t)$ is the fundamental solution of $\dot R(t)=-l_\sigma^{-2}J\nabla^2_x H(t/l_\sigma,0,0)R(t)$ with $R_0(0)=I_2$.
\end{lem}

\begin{proof}
Since $H(dz+x_1dx_2)=\Phi^*\lambda_F$ in $U$, then we have the following local model
\begin{equation}\label{equ: local model of fs}
\Phi: (\mathbb R/\mathbb Z\times D^2_\epsilon, H_s(dz+x_1dx_2))\cong(U,e^{f_s}\lambda_F),\quad H_s=e^{f_s\circ\Phi}H.
\end{equation}
In the following proof, we identify $U$ with $\mathbb R/\mathbb Z\times D^2_\epsilon$. By assumption, we have
\begin{equation}\label{equ: properties of fs and H}
\begin{aligned}
&\mathrm{supp}f_s=\mathrm{supp}h\subset S^*_F\pi(U_1),\quad H|_{U_1}\equiv l_\sigma\equiv H|_{\mathbb R/\mathbb Z\times\{0\}},\\
&\quad \quad f_s|_{\gamma_\sigma}=0=h|_{\gamma_\sigma},\quad df_s|_{\gamma_\sigma}= 0= dH|_{\mathbb R/\mathbb Z\times\{0\}}.
\end{aligned}
\end{equation}
Then we obtain that $H_s=e^{f_s\circ\Phi}l_\sigma$ in $U_1$, $H_s|_{\mathbb R/\mathbb Z\times\{0\}}=l_\sigma$ and $dH_s|_{\mathbb R/\mathbb Z\times\{0\}}=0$.
Let $\alpha_s=H_s\alpha_{st}$, where $\alpha_{st}:=dz+x_1dx_2$. We compute the Reeb vector field of $\alpha_s$ as
\begin{equation}\label{equ: formula of reeb vf}
\begin{aligned}
R_{\alpha_s}:&=(R_{\alpha_s,z},R_{\alpha_s,x_1},R_{\alpha_s,x_2})\\
&=H_s^{-2}(H_s+x_1 H_{s,x_1}',H_{s,x_2}'-x_1H_{s,z}',-H_{s,x_1}').
\end{aligned}
\end{equation}
Let $x_s=(x_{s,1},x_{s,2})$ and let $(z_s(t),x_s(t))$ be an integral curve of $R_{\alpha_s}$ satisfying $\dot x_s(t)=H_s^{-2}(H_{s,x_2}'-x_1H_{s,z}',-H_{s,x_1}')(z_s(t),x_s(t)),\forall t\in\mathbb{R}$.
Linearize this equation along $\mathbb{R}/\mathbb{Z}\times\{0\}$ in $(x_1,x_2)$-direction, we obtain that
\begin{equation}\label{equ: linearized Reeb flow}
\begin{aligned}
\dot R_s(t)&=l_\sigma^{-2}\begin{pmatrix}(H_{s})_{x_1x_2}'' & (H_{s})_{x_2x_2}''\\ -(H_{s})_{x_1x_1}'' & -(H_{s})_{x_1x_2}''\end{pmatrix}(t/l_\sigma,0,0)R_s(t)\\
&=l_\sigma^{-2}J\bigg(-\nabla_x^2H+\frac{sl_\sigma}{2}\nabla^2_x\tilde h\bigg)(t/l_\sigma,0,0)R_s(t),\quad R_s(0)=I_2.
\end{aligned}
\end{equation}
The last equality follows from (\ref{equ: properties of fs and H}) and the following identities
\begin{gather*}
\nabla_x^2H_s(z,0,0)=(\nabla_x^2H+l_\sigma\nabla_x^2(f_s\circ\Phi))(z,0,0),\\
\nabla_x^2(f_s\circ\Phi)(z,0,0)=-\frac{s}{2}\nabla_x^2\tilde h(z,0,0),\quad \forall z\in \mathbb R/\mathbb Z.
\end{gather*}
By taking the derivative of (\ref{equ: linearized Reeb flow}) at $s=0$, we have
\begin{equation}\label{equ: differential of Rs}
\partial_s \dot R_s(t)\big|_{s=0}=l_\sigma^{-2}\left(\frac{l_\sigma}{2}J\nabla^2_x\tilde h(t/l_\sigma,0,0)R_0(t)-J\nabla^2_xH(t/l_\sigma,0,0) \partial_s R_s(t)\big|_{s=0}\right).
\end{equation}
From the assumption and (\ref{equ: properties of fs and H}), we know that
\begin{equation}\label{equ: hessen of H and tilde h is 0}
\begin{aligned}
\nabla^2_x H(t/l_\sigma,0,0)&=0,\quad \forall t\in(l_\sigma/2-\epsilon,l_\sigma/2+\epsilon),\\ 
\nabla^2_x\tilde h(t/l_\sigma,0,0)&=0,\quad \forall t\notin (l_\sigma/2-\epsilon,l_\sigma/2+\epsilon).
\end{aligned}
\end{equation}
This implies that $R_0(t)\equiv R_0(l_\sigma/2),\forall t\in(l_\sigma/2-\epsilon,l_\sigma/2+\epsilon)$. Moreover, for every $s\in(-\epsilon,\epsilon)$,
\begin{equation}\label{equ: formula of Rs}
R_s(t)=\left\{\begin{aligned}
&R_0(t),\quad &\forall t\in[0,l_\sigma/2-\epsilon],\\
&R_0(t)R_0(l_\sigma/2+\epsilon)^{-1}R_s(l_\sigma/2+\epsilon),\quad & \forall t\in[l_\sigma/2+\epsilon,1].
\end{aligned}\right.
\end{equation}

By definition, the linearized Poincar\'e map is $P_{\sigma}(f_s)=R_s(1),\forall s\in(0,\epsilon)$. In particular, $P_\sigma(0)=R_0(1)$. Let $C:={(2l_\sigma)}^{-1}\int^{l_\sigma/2+\epsilon}_{l_\sigma/2-\epsilon}J\nabla^2_x\tilde h(t/l_\sigma,0,0)R_0(t)dt$. By taking the derivative of $P_{\sigma}(f_s)$ at $s=0$, we deduce that
\begin{equation}\label{equ: differential of P}
\begin{aligned}
\quad\ dP_{\sigma}(h)\big|_{f=0}&=\partial_s R_s(l_\sigma)\big|_{s=0}=\partial_s R_s(l_\sigma)\big|_{s=0}-\partial_s R_s(0)\big|_{s=0}=\int^{l_\sigma}_0\partial_s \dot R_s(t)\big|_{s=0}dt&\\
&=l_\sigma^{-2}\int^{l_\sigma}_0\frac{l_\sigma}{2}J\nabla^2_x\tilde h(t/l_\sigma,0,0)R_0(t)-J\nabla^2_xH(t/l_\sigma,0,0)\partial_s R_s(t)\big|_{s=0}dt&\\
&=C-l_\sigma^{-2}\int^{l_\sigma}_{l_\sigma/2+\epsilon}J\nabla^2_xH(t/l_\sigma,0,0)\partial_s R_s(t)\big|_{s=0}dt,
\end{aligned}
\end{equation}
where $R_s(0)\equiv I_2,\forall s$, (\ref{equ: hessen of H and tilde h is 0}) and (\ref{equ: formula of Rs}) are used. Using (\ref{equ: formula of Rs}), (\ref{equ: linearized Reeb flow}) and (\ref{equ: hessen of H and tilde h is 0}), we obtain
\begin{equation}\label{equ: Rs at t0+e}
\begin{aligned}
&\quad\ R_s(l_\sigma/2+\epsilon)-R_s(l_\sigma/2-\epsilon)=\int^{l_\sigma/2+\epsilon}_{l_\sigma/2-\epsilon}\dot R_s(t)dt\\
&=l_\sigma^{-2}\int^{l_\sigma/2+\epsilon}_{l_\sigma/2-\epsilon}J(\frac{sl_\sigma}{2}\nabla^2_x\tilde h-\nabla^2_xH)(t/l_\sigma,0,0)R_s(t)dt\\
&=\int^{l_\sigma/2+\epsilon}_{l_\sigma/2-\epsilon}\frac{s}{2l_\sigma}J\nabla^2_x\tilde h(t/l_\sigma,0,0)R_s(t)dt.
\end{aligned}
\end{equation}
Then, by taking the differential of (\ref{equ: Rs at t0+e}) at $s=0$, we have
$$
\partial_s R_s(l_\sigma/2+\epsilon)\big|_{s=0}=\frac{1}{2l_\sigma}\int^{l_\sigma/2+\epsilon}_{l_\sigma/2-\epsilon}J\nabla^2_x\tilde h(t/l_\sigma,0,0)R_0(t)dt=C.
$$
From this relation and (\ref{equ: formula of Rs}), we conclude that
\begin{equation}\label{equ: differential of Rs after t0+e}
\begin{aligned}
\partial_s R_s(t)|_{s=0}&=R_0(t)R_0(l_\sigma/2+\epsilon)^{-1}\partial_s R_s(l_\sigma/2+\epsilon)\big|_{s=0}\\
&=R_0(t)R_0(l_\sigma/2+\epsilon)^{-1}C,\quad \forall t\in[l_\sigma/2+\epsilon,l_\sigma].
\end{aligned}
\end{equation}
Finally, we have
\begin{equation*}
\begin{aligned}
dP_{\sigma}(h)|_{f=0}&=C+\int^{l_\sigma}_{l_\sigma/2+\epsilon}-l_\sigma^{-2}J\nabla^2_xH(t/l_\sigma,0,0)R_0(t)R_0(l_\sigma/2+\epsilon)^{-1}Cdt\\
&=C+(R_0(l_\sigma)-R_0(l_\sigma/2+\epsilon))R_0(l_\sigma/2+\epsilon)^{-1}C=R_0(l_\sigma)R_0(l_\sigma/2+\epsilon)^{-1}C\\
&=P_\sigma(0)R_0(l_\sigma/2)^{-1}\bigg(\frac{1}{2l_\sigma}\int^{l_\sigma/2+\epsilon}_{l_\sigma/2-\epsilon}J\nabla^2\tilde h(t/l_\sigma,0,0)dt\bigg)R_0(l_\sigma/2),
\end{aligned}
\end{equation*}
where (\ref{equ: differential of P}), (\ref{equ: differential of Rs after t0+e}) and (\ref{equ: linearized Reeb flow}) are used. The proof is completed.
\end{proof}
As a consequence of Lemma \ref{lem: differential of P_gamma(f)}, we prove the following lemma.
\begin{lem}\label{lem: genericity of nondegenercy 1}
Let $F\in\mathcal{F}(\Sigma)$ and $\sigma_1,\cdots,\sigma_k\in \mathrm{CG}^F$ be pairwise distinct. For any $\epsilon>0$ sufficiently small, there exists a $F'\in \mathcal{F}(\Sigma)$ such that $\|F-F'\|_{C^{\infty}}<\epsilon$ and $\sigma_1,\cdots,\sigma_k\in\mathrm{CG}^{F'}$ are nondegenerate. Moreover, if $F$ belongs to $\mathcal{F}^s(\Sigma)$ or $\mathcal{M}(\Sigma)$, then $F'$ belongs to the same class.
\end{lem}
\begin{proof}
Assume $l_{\sigma_i}/2\in \mathrm{reg}(\sigma_i),\forall i=1,\cdots,k$. For each $\sigma_i$, we can choose a small open neighborhood $U_{i,1}$ of $\gamma_{\sigma_i}(l_{\sigma_i}/2)$ in its small tubular neighborhood $U_i$ such that $\gamma_{\sigma_i}(l_{\sigma_i}/2-\epsilon,l_{\sigma_i}/2+\epsilon)\subset U_{i,1}$, $\pi(U_{i,1})$ are pairwise disjoint and $\pi(U_{i,1})\cap \sigma_j=\emptyset,\forall i\neq j$. Let $(q,p)$ be the local coordinates in $U_i$. Choose
$f_s=-\frac{1}{2}\ln\left(1+s(A(q)p,p)\right)\in \mathcal{F}^{pert}_{\mathcal{U}_1}$, where $\mathcal{U}_1=\cup_{i=1}^k \pi(U_{i,1})$ and the support of $A(\cdot)$ lies in $\mathcal{U}_1$ with the following form
\begin{equation}\label{equ: Choice of tilde A}
\big(\tilde A_{ij}(t_1,t_2)\big)_{2\times 2}=\begin{pmatrix}a_{i,1}(t_2)&a_{i,2}(t_2)t_1\\a_{i,2}(t_2)t_1& a_{i,3}(t_2)t_1^2\end{pmatrix},\quad a_{i,\{1,2,3\}}\in C^\infty_0((l_{\sigma_i}/2-\epsilon,l_{\sigma_i}/2+\epsilon),\mathbb{R}),
\end{equation}
as the representation (\ref{equ: local representation of tilde A}). Since the projected domains $\pi(U_{i,1}),\forall i$, are pairwise disjoint and each $\pi(U_{i,1})$ avoids $\{\sigma_j\}_{j\neq i}$ as well as self-intersections of $\sigma_i$, we can independently choose \eqref{equ: Choice of tilde A} for each $i$, so that the support of $f_s$ does not intersects $\sigma_i(l_{\sigma_i}/2-\epsilon,l_{\sigma_i}/2+\epsilon)$. By Proposition \ref{prop: condition of perturbation metrix tilde A}, we also have $f_{s}|_{\gamma_{\sigma_i}}\equiv 0\equiv df_{s}|_{\gamma_{\sigma_i}}$. Let $h=\partial_sf_{s}|_{s=0}$. It follows from \eqref{equ: hessen of h} that
\begin{equation}\label{equ: hessen of h at sigma i}
\nabla^2 h|_{\xi_F}(\gamma_{\sigma_i}(t))=\begin{pmatrix}-a_{i,1}(t) & -a_{i,2}(t)\\ -a_{i,2}(t) & -a_{i,3}(t)\end{pmatrix},
\end{equation}
under the basis $\{V^i_{t_1,t_2},W^i_{t_1,t_2}\}$ defined in \eqref{equ: basis bar V and bar W}. Due to the arbitrariness of $a_{i,\{1,2,3\}}$, we can rewrite $\nabla^2 h|_{\xi_F}(\gamma_{\sigma_i}(t))$ in basis $\xi_F|_{\gamma_{\sigma_i}}=\{\partial_{x_1},\partial_{x_2}\}$ with arbitrary elements in $C^\infty_0((l_{\sigma_i}/2-\epsilon,l_{\sigma_i}/2+\epsilon),\mathbb{R})$. Then applying Lemma \ref{lem: differential of P_gamma(f)}, we conclude that $dP_{\sigma_i}|_{f=0}$ is surjective. Hence, we can perturb $F$ to $\hat F=\hat F^*\circ \mathcal L_{F}$ with $\hat F^*=e^{-f_\epsilon}F^*$ so that $\{\sigma_1,\cdots, \sigma_k\}$ are all nondegenerate in $\mathrm{CG}^{\hat F}$. Finally, as the multiplier $e^{-f_\epsilon}$ preserves the structure of symmetric Finsler or Riemannian metrics, it follows that if $F\in\mathcal{F}^s(\Sigma)$ or $\mathcal{M}(\Sigma)$, then $\hat F$ belongs to the same class. The lemma holds.
\end{proof}
Let $\mathcal{F}^N_{F,C^r}$ denote the collection of all $C^r$-smooth $[0,1]^N$-families in $\mathcal{F}(\Sigma)$ with $F$ at $0\in[0,1]^N$ for any $N>0$. Similarly, we have $\mathcal{F}^{s,N}_{F,C^r}$ and $\mathcal{M}^N_{g,C^r}$ for symmetric Finsler and Riemannian metrics. For any $\{F_{\tau}\}_{\tau\in[0,1]^N}$ in $\mathcal{F}^N_{F,C^r}$, we obtain a corresponding $[0,1]^N$-family of contact forms $e^{f_\tau}\lambda_F$ in $\Lambda^{Fin}_F$, where $f_{\tau}=-\ln F^*_{\tau}$, due to Proposition \ref{prop: properties of Finsler geodesic and Reeb flows}-(1). By employing the strategy presented in Lemma 3.2 of \cite{Irie2018} within the framework of Finsler and Riemannian setting, we conclude the following result.

\begin{thm}\label{lem: genericity of nondegenercy 2}
Given any $F\in\mathcal{F}(\Sigma)$ and $N>0$ sufficiently large. For generic $F_{[0,1]^N}\in \mathcal{F}^N_{F,C^\infty}$, we have a full measure subset $\{\tau\in[0,1]^N: F_\tau\ \mathrm{is\ bumpy}\}\subset[0,1]^N$. If $F\in \mathcal{F}^s(\Sigma)$ or $F^2=g\in\mathcal{M}(\Sigma)$, a similar statement holds for generic $F_{[0,1]^N}\in \mathcal{F}^{s,N}_{F,C^\infty}$ or generic $g_{[0,1]^N}\in\mathcal{M}^{N}_{g,C^\infty}$.
\end{thm}

Before we prove this theorem, we need the following lemmas. Define a $C^{r}$-map
$$\begin{aligned}
E:\mathcal{F}^N_{F,C^r}\times S^*_F\Sigma\times [0,1]^N\times \mathbb{R}_+&\rightarrow S^*_F\Sigma\times S^*_F\Sigma,\\
\big(F_{[0,1]^N},(q_0,p_0),\tau,l\big)&\mapsto \big((q_0,p_0),(q_\tau(l),p_\tau(l))\big),
\end{aligned}$$
where $(q_\tau(l),p_\tau(l))=\varphi^l_{F_\tau^*}(q_0,p_0)$, i.e. the $l$-image of the Hamiltonian flow $\varphi^t_{F_\tau^*}$. Similarly, we have $C^{r}$-maps $E^s$ and $E^{R}$ by replacing $\mathcal{F}^N_{F,C^r}$ to $\mathcal{F}^{s,N}_{F,C^r}$ and $\mathcal{M}^N_{g,C^r}$, respectively. Let $\Delta_F$ be the diagonal of $S^*_F\Sigma\times S^*_F\Sigma$ and let $\mathrm{CG}^F_p$ be the set of closed geodesics on $\Sigma$ with minimal period. Let $\mathrm{Dom}E$ be the domain of map $E$. Define an open subset in $\mathrm{Dom}E$ as
\begin{equation}\label{equ: open set O}
\begin{aligned}
\mathcal{O}:=&\big\{(F_{[0,1]^N},(q_0,p_0),\tau,l)\in\mathrm{Dom}E:\\
&\quad(q_\tau(l_1),p_\tau(l_1))\neq(q_\tau(l_2),p_\tau(l_2)),\forall l_1\neq l_2\ \mathrm{in}\ [0,\frac{l}{2}]\big\}.
\end{aligned}
\end{equation}
and define $\mathcal{O}^s\subset \mathrm{Dom}E^s$ and $\mathcal{O}^{R}\subset \mathrm{Dom}E^R$, similarly. Denote the restricted preimage as
\begin{equation*}
\begin{aligned}
\mathrm{CG}_p(\mathcal{F}^N_{F,C^r}):= E|_{\mathcal{O}}^{-1}(\Delta_F)&=\left\{(F_{[0,1]^N},(q_0,p_0),\tau,l)\in\mathcal{O}:\ (q_0,p_0)=(q_\tau(l),p_\tau(l))\right\}\\
&\cong\left\{(F_{[0,1]^N},\tau,\sigma): F_{[0,1]^N}\in\mathcal{F}^N_{F,C^r}, \tau\in[0,1]^N, \sigma\in \mathrm{CG}^{F_\tau}_p\right\}.
\end{aligned}
\end{equation*}
Let $\Delta_g$ denote the diagonal of $S^*_g\Sigma\times S^*_g\Sigma$. To understand $\mathrm{CG}_p(\mathcal{F}^N_{F,C^r})$, $\mathrm{CG}_p(\mathcal{F}^{s,N}_{F,C^r}):=E^s|_{\mathcal{O}^s}^{-1}(\Delta_F)$ and $ \mathrm{CG}_p(\mathcal{M}^{N}_{g,C^r}):=E^R|_{\mathcal{O}^R}^{-1}(\Delta_g)$, we provide the following lemma.
\begin{lem}\label{claim 1}
$E|_{\mathcal{O}}$, $E^s|_{\mathcal{O}^s}$ and $E^R|_{\mathcal{O}^R}$ are all transverse to $\Delta_F$.
\end{lem}
\begin{proof}
We first establish that $E|_{\mathcal{O}}$ is transversal to $\Delta_F$. Let $\mathrm{Im}E$ be the image of map $E$. For every $(x_0,x_0)\in\mathrm{Im}E|_{\mathcal{O}}\cap \Delta_F$, we have $x_0=\varphi^l_{F_\tau^*}(x_0)$. Then $T_{(x_0,x_0)}\Delta_F=\{(v,v): \forall v\in T_{x_0}S^*_F\Sigma\}$, and
\begin{equation}\label{equ: tangent space of ImE}
\begin{aligned}
T_{(x_0,x_0)}\mathrm{Im}(E|_{\mathcal{O}})&=\left\{\big(v,d\varphi^l_{F_\tau^*}(v)\big|_{x_0}\big):\ \forall v\in T_{x_0}S^*_F\Sigma\right\}+\left\{\big(0,\frac{d}{dt}\varphi^t_{F_\tau^*}(x_0)\big|_{t=l}\big)\right\}\\
&+\left\{\big(0,\frac{d}{ds}\varphi^l_{(F^s_\tau)^*}(x_0)\big|_{s=0}\big):\ \forall \big\{F^s_{[0,1]^N}\big\}_{s\in(-\epsilon,\epsilon)}\subset \mathcal{F}^N_{F,C^r}\right\}\\
&+\left\{\big(0,\frac{d}{d\tau_i}\varphi^l_{F_\tau^*}(x_0)\big|_{\tau_i=0}\big):\ \forall i\in\{1,\cdots, N\}\right\}.
\end{aligned}
\end{equation}
By identifying $(S^*_F\Sigma,e^{f_\tau}\lambda_F)$ to $(S^*_{F_\tau}\Sigma,\lambda_{F_\tau})$ with $f_\tau=-\ln F_\tau^*$, we obtain a Reeb orbit $x_0(\mathbb{R})$ in $(S^*_{F_\tau}\Sigma,\lambda_{F_\tau})$ starting from $x_0$ with minimal period $l$. The projection $\pi\circ x_0(\mathbb R)$ then becomes a closed geodesic on $(\Sigma,F_\tau)$. Let $U_0$ be an open small tubular neighborhood of $x_0(\mathbb{R})$ with local coordinates $(q_\tau,p_\tau)$. Choose a small open neighborhood $U_{0,1}\subset U_0$ so that
\begin{equation}\label{equ: definition of U01}
x_0|_{(l/2-\epsilon,l/2+\epsilon)}\subset U_{0,1},\quad (l/2-\epsilon,l/2+\epsilon)\subset \mathrm{reg}(\pi\circ x_0(\mathbb{R})).
\end{equation}
Then we consider a path $\lambda^{s}_{\tau}:=e^{f_s}\lambda_{F_\tau}\in \Lambda^{Fin}_{F_\tau},s\in(-\epsilon,\epsilon)$,
where $f_s=-\frac{1}{2}\ln\left(1+s(A(q_\tau)p_\tau,p_\tau)\right)\in \mathcal{F}_{\pi(U_{0,1})}^{pert}$ and $A(\cdot)$ supports in $\pi(U_{0,1})$ under the local coordinates $(q_\tau,p_\tau)$. The initial derivative of $f_s$ is $h(q_\tau,p_\tau)=\partial_s f_s(q_\tau,p_\tau)\big|_{s=0}=-\frac{1}{2}(A(q_\tau)p_\tau,p_\tau)$.
Under the local representation (\ref{equ: local representation of tilde A}), we choose $A(\cdot)$ as
\begin{equation}\label{equ: choice of tilde A in U01}
\tilde A(q_\tau(t_1,t_2))=\begin{pmatrix}b_1(t_2)&-b_3(t_2)\\-b_3(t_2)&2b_2(t_2)t_1\end{pmatrix},\quad b_{1,2,3}\in C^\infty_0\big((l/2-\epsilon,l/2+\epsilon),\mathbb{R}\big).
\end{equation}
For every $t\in \big(l/2-\epsilon,l/2+\epsilon\big)$, we have
\begin{equation}\label{equ: h of x_0}
h(x_0(t))=-\frac{1}{2}\left(A(q_\tau(t))p_\tau(t),p_\tau(t)\right)=-b_2(t)t_1\big|_{t_1=0}=0,
\end{equation}
where $x_0(t)=(q_\tau(t),p_\tau(t))$ in $U_{0,1}$. Moreover, for any $t\notin(l/2-\epsilon,l/2+\epsilon)$, we have $h(x_0(t))=0$. Under the symplectic basis $\{V^\tau_{t_1,t_2},W^\tau_{t_1,t_2}\}$ of $(\xi_{F_\tau},d\lambda_{F_\tau})$, see (\ref{equ: basis bar V and bar W}), we obtain from (\ref{equ: gradient of h}) that
\begin{equation}\label{equ: X dh is}
\nabla h(x_0(t))|_{\xi_{F_\tau}}=(b_3(t),-b_2(t)),\quad X_{h}|_{x_0(t)}=(b_2(t),b_3(t)),\quad \forall t\in\big(l/2-\epsilon,l/2+\epsilon\big),
\end{equation}
where $\iota_{X_{h}}d\lambda_{F_\tau}=-dh$ on $\xi_{F_\tau}$.

Now consider the Reeb vector field $\mathcal{R}_s=e^{-f_s}(\mathcal{R}_0-X_{f_s})$ of $\lambda^s_{\tau},\forall s\in(-\epsilon,\epsilon)$, where $X_{f_s}\in\xi_{F_\tau}$ is uniquely determined by $\iota_{X_{f_s}}d\lambda_{F_\tau}=-df_s$  on $\xi_{F_\tau}$. 
Let $x_s(t):=\varphi^t_{\lambda^s_{\tau}}(x_0)$ solving
\begin{equation}\label{equ: Reeb flow of lambda s}
\dot x_s(t)=\mathcal{R}_s(x_s(t))=e^{-f_s(x_s(t))}(\mathcal{R}_0(x_s(t))-X_{f_s}(x_s(t))),\quad \forall s\in (-\epsilon,\epsilon).
\end{equation}
By differentiating (\ref{equ: Reeb flow of lambda s}) in $s$ at $s=0$, then using $f_0=0$, $X_{f_0}=0$ and
\begin{equation}\label{equ: X of dh}
\iota_{X_{h}}d\lambda_{F_\tau}=-\frac{\partial df_s}{\partial s}\big|_{s=0}=-dh,\quad X_{h}:=\frac{\partial X_{f_s}}{\partial s}\big|_{s=0}\in \xi_{F_\tau},
\end{equation}
we obtain that
\begin{equation}\label{equ: surjection 1}
\dot\eta(t)=-h(x_0(t))\mathcal{R}_0(x_0(t))-X_{h}(x_0(t))+\nabla \mathcal{R}_0(\eta(t)),\quad \eta(t):=\frac{\partial x_s(t)}{\partial s}\big|_{s=0}.
\end{equation}
In the Martinet tube $\Phi_\tau:\mathbb R/\mathbb Z\times D^2_\epsilon\rightarrow U_0 ((z,x_1,x_2)\mapsto (q_\tau,p_\tau))$, we have $\Phi_{\tau}^*\lambda_{F_\tau}=H_\tau(dz+x_1dx_2)$ and $x_0(t)=\Phi_\tau(t/l,0,0)$, where
\begin{equation}\label{equ: properties of H tau}
H_\tau|_{\Phi_\tau^{-1}(U_{0,1})}\equiv l,\quad H_\tau|_{\mathbb R/\mathbb Z\times\{0\}}=l,\quad dH_\tau|_{\mathbb R/\mathbb Z\times\{0\}}=0.
\end{equation}
Under the frame $\{d\Phi_\tau\partial_{x_1},d\Phi_\tau\partial_{x_2}\}$ of $\xi|_{x_0(\mathbb R)}$, we can further rewrite $X_{h}|_{x_0(t)}=(\tilde b_2(t),\tilde b_3(t))$, where $\tilde b_{2},\tilde b_{3}\in C^\infty_0((l/2-\epsilon,l/2+\epsilon),\mathbb{R})$.

For convenience, we identify $(U_0,\lambda_{F_\tau})$ and $(\mathbb R/\mathbb Z\times D^2_\epsilon,H_\tau(dz+x_1dx_2))$. As (\ref{equ: formula of reeb vf}), we can rewrite $\mathcal{R}_0=H_\tau^{-2}(H_\tau+x_1H_{\tau,x_1}',H_{\tau,x_2}'-x_1H_{\tau,z}',-H_{\tau,x_1}')$. $\eta$ becomes a vector field along $\mathbb R/\mathbb Z\times\{0\}$. Using \eqref{equ: properties of H tau}, we compute
\begin{equation}\label{equ: nabla R0}
\begin{aligned}
\nabla \mathcal{R}_0\big|_{(t/l,0,0)}(\eta(t))&=l^{-2}\left(0,d(H_{\tau,x_2}')|_{(t/l,0,0)}\eta(t),-d(H_{\tau,x_1}')|_{(t/l,0,0)}\eta(t)\right)\\
&=l^{-2}\left(0,H_{\tau,x_2x_1}''\eta_{x_1}+H_{\tau,x_2x_2}''\eta_{x_2},-H_{\tau,x_1x_1}''\eta_{x_1}-H_{\tau,x_1x_2}''\eta_{x_2}\right)|_{(t/l,0,0)}(t)\\
&=l^{-2}(0,(-J\nabla^2_{x}H_\tau(t/l,0,0)\eta_{x}(t))^T),\quad \eta_x(t):=(\eta_{x_1},\eta_{x_2})^T(t).
\end{aligned}
\end{equation}
Then the expressions (\ref{equ: h of x_0}), (\ref{equ: X dh is}) and (\ref{equ: nabla R0}) enable us to rephrase (\ref{equ: surjection 1}) as
$$\dot{\eta}_z(t)=0 \quad \text{and}\quad \dot{\eta}_x(t)=-\tilde b(t)-l^{-2}J\nabla^2_{x}H_\tau(t/l,0,0){\eta}_{x}(t),\quad \tilde b(t):=(\tilde b_2,\tilde b_3)^T(t).$$
Since $\eta(0)=0$, we solve that
\begin{equation}\label{equ: formula of eta x}
\eta_z(t)=0,\quad \eta_x(t)=-R_{\tau}(t)\int^t_0R_\tau(s)^{-1}\tilde b(s)ds,
\end{equation}
where $R_\tau$ solves
$\dot R_\tau(t)=-l^{-2}J\nabla^2_{x}H_\tau(t/l,0,0)R_\tau(t)$ and $R_\tau(0)=I_2$.
Moreover, (\ref{equ: properties of H tau}) implies that
$$\nabla^2_xH_\tau(t/l,0,0)=0,\quad R_\tau(t)=R_\tau(l/2),\quad \forall t\in(l/2-\epsilon,l/2+\epsilon),$$
then (\ref{equ: formula of eta x}) and  implies that $\eta_z(l)=0$ and
\begin{equation}\label{equ: tilde eta x}
\eta_x(l)=-R_{\tau}(l)R_\tau(l/2)^{-1}\int^{l/2+\epsilon}_{l/2-\epsilon}\tilde b(s)ds.
\end{equation}
The arbitrariness of $\tilde b$ implies that $\eta_x(l)$ can be any element in $\xi_{F_\tau}|_{x_0}$. Now we define $(F^s_\tau)^*(q_\tau,p_\tau)=\big(F_\tau(q_\tau,p_\tau)^2+s(A(q_\tau)p_\tau,p_\tau)\big)^{\frac{1}{2}}$ for every $(q_\tau,p_\tau)\in T^*\pi(U_{0,1})$ and $(F^s_\tau)^*=F^*_\tau$ elsewhere, where $A(\cdot)$ is described in (\ref{equ: choice of tilde A in U01}). Due to Proposition \ref{prop: properties of Finsler geodesic and Reeb flows}, we have $\frac{d}{ds}\varphi^l_{(F^s_\tau)^*}(x_0)\big|_{s=0}=\frac{d}{ds}\varphi^l_{\lambda^s_\tau}(x_0)\big|_{s=0}=\eta(l)=(\eta_z,\eta_x)(l)$.
Finally, by (\ref{equ: tangent space of ImE}) and the flexibility of $\eta_x(l)$, we obtain that $d(E|_\mathcal{O})|_{(x_0,x_0)}$ is surjective and
$$T_{(x_0,x_0)}\mathrm{Im}(E|_{\mathcal{O}})+T_{(x_0,x_0)}\Delta_F=T_{x_0}S^*_F\Sigma\times T_{x_0}S^*_F\Sigma.$$

If $F_\tau\in \mathcal F^s(\Sigma)$ or $g_\tau=F_\tau^2\in\mathcal{M}(\Sigma)$, then the above variation $F_\tau^s$ still belongs to the same class. By incorporating Proposition \ref{prop: properties of Finsler geodesic and Reeb flows} or Proposition \ref{prop: properties of Riemannian geodesic and Reeb flows} and the aforementioned arguments, we can prove the surjectivity of $E^s|_{\mathcal{O}^s}$ (resp. $E^R|_{\mathcal{O}^R}$), as well as the transversality between $E^s|_{\mathcal{O}^s}$ (resp. $E^R|_{\mathcal{O}^R}$) and $\Delta_F$ (resp. $\Delta_g$). Hence, the proof is completed.
\end{proof}
As a conclusion of Lemma \ref{claim 1} and Lemma \ref{lem: Appendix Sard-Smale theorem 2}, we conclude that $\mathrm{CG}_p(\mathcal{F}^N_{F,C^r})$ is a $C^{r}$-submanifold of $\mathrm{Dom}E$ with codimension $3$. Since the projection from $\mathrm{Dom}E$ to $\mathcal{F}^N_{F,C^r}$ is Fredholm with index $N+4$,
and each Finsler surface admits a closed geodesic, then its $3$-codimensional restriction
\begin{equation}\label{equ: def of P1}
P_1:\mathrm{CG}_{p}(\mathcal{F}^N_{F,C^r})\rightarrow \mathcal{F}^N_{F,C^r}\ \big((F_{[0,1]^N},\tau,\sigma)\mapsto F_{[0,1]^N}\big)
\end{equation}
is surjective with Fredholm index $N+1$.

Now, we consider a $\mathrm{Sp}(2)$-bundle $\mathcal{E}\rightarrow \mathrm{CG}_{p}(\mathcal{F}^N_{F,C^r})$ with fiber $$\mathcal{E}_{(F_{[0,1]^N},\tau,\sigma)}=\mathrm{Sp}(\xi_{\gamma_\sigma(0)},d\lambda_{F_\tau})\cong \mathrm{Sp}(2).$$
We also denote by $\mathcal{E}^s,\mathcal{E}^R$ the subbundles of $\mathcal{E}$ over $\mathrm{CG}_{p}(\mathcal{F}^{s,N}_{F,C^r})$, $\mathrm{CG}_{p}(\mathcal{M}^N_{g,C^r})$. According to Theorem~1.7.3 of \cite{Long2002}, any element $M\in\mathrm{Sp}(2)$ is symplectically similar to one of the following matrices
$$N(\pm 1,b)=\begin{pmatrix}\pm 1& b\\0 & \pm 1\end{pmatrix},\ D(a)=\begin{pmatrix}a & 0\\ 0 & a^{-1}\end{pmatrix},\ R(\theta)=\begin{pmatrix}\cos\theta&-\sin\theta\\ \sin\theta & \cos\theta\end{pmatrix},$$
where $b\in\{0,\pm1\},a\in\mathbb{R}\setminus\{0,\pm1\}$ and $\theta\in(0,\pi)\cup(\pi,2\pi)$.
Denote $\mathrm{Sp}_{\pm 1,b}(2)$, $\mathrm{Sp}_a(2)$ and $\mathrm{Sp}_\theta(2)$ as the collection of $M\in\mathrm{Sp}(2)$ associating to $N(\pm 1,b)$, $D(a)$ and $R(\theta)$, respectively. $\mathrm{Sp}_{\pm1,0}(2)$ only include one point, $\mathrm{Sp}_{\pm 1,b\neq 0}(2)$, $\mathrm{Sp}_a(2)$ and $\mathrm{Sp}_\theta(2)$ are 2-dimensional submanifolds of $\mathrm{Sp}(2)$. Denote $\mathcal{E}_{\pm1,b}$, $\mathcal{E}_a$ and $\mathcal{E}_\theta$ as the subbundles over $\mathrm{CG}_{p}(\mathcal{F}^N_{F,C^r})$ with fiber $\mathrm{Sp}_{\pm 1,b}(2)$, $\mathrm{Sp}_a(2)$ and $\mathrm{Sp}_{\theta}(2)$. Moreover, the linear Poincar\'e map induces a $C^{r-1}$-section of $\mathcal{E}$:
$$\Gamma: \mathrm{CG}_{p}(\mathcal{F}^N_{F,C^r})\rightarrow \mathcal{E},\ (F_{[0,1]^N},\tau,\sigma)\mapsto P_{\sigma}\in \mathrm{Sp}(\xi_{F}|_{\gamma_\sigma(0)},d\lambda_{F_\tau}).$$
The projection of the intersection $\mathrm{Im}\Gamma\cap \mathcal{E}_{*}$ is defined as
$$\mathrm{CG}_{p,*}(\mathcal{F}^N_{F,C^r}):=\left\{(F_{[0,1]^N},\tau,\sigma)\in \mathrm{CG}_{p}(\mathcal{F}^N_{F,C^r}): P_{\sigma}\in \mathrm{Sp}_{*}(2)\right\},$$
for $*=(\pm1,b),a,\theta$, respectively.

Similarly, we obtain $\mathcal{E}^s_*$, $\mathcal{E}^R_*$ for $*=(\pm1,b),a,\theta$ and the restricted sections $\Gamma^s$ and $\Gamma^R$ of $\Gamma$ on subbundles $\mathcal{E}^s$ and $\mathcal{E}^R$. The sets $\mathrm{CG}_{p,*}(\mathcal{F}^{s,N}_{F,C^r})$, $\mathrm{CG}_{p,*}(\mathcal{M}^N_{g,C^r})$ denote the projection of $\mathrm{Im}\Gamma^s\cap \mathcal{E}^s$ and $\mathrm{Im}\Gamma^R\cap \mathcal{E}^R$. To further understand $\mathrm{CG}_{p,*}(\mathcal{F}^N_{F,C^r})$, $\mathrm{CG}_{p,*}(\mathcal{F}^{s,N}_{F,C^r})$ and $\mathrm{CG}_{p,*}(\mathcal{M}^N_{g,C^r})$, we establish the following lemma.
\begin{lem}\label{claim 2}
The section $\Gamma$ is transverse to $\mathcal{E}_*$ in $\mathcal{E}$ for every $*\in\{(\pm1,b),a,\theta\}$ and every $b\in\{0,\pm 1\}$, $a\in\mathbb{R}\setminus\{0,\pm1\}$ and $\theta\in(0,\pi)\cup(\pi,2\pi)$. Moreover, a similar statement holds for $\Gamma^s,\mathcal{E}^s_*$ in $\mathcal{E}^s$ and for $\Gamma^s,\mathcal{E}^R_*$ in $\mathcal{E}^R$, respectively.
\end{lem}
\begin{proof}
Indeed, without loss of generality, let $(F_{[0,1]^N},\tau,\sigma)$ be an arbitrary element in $\mathrm{Im}\Gamma\cap \mathcal{E}_{\theta}$ with $P_{\sigma}\in \mathrm{Sp}_\theta(2)$. Perturb $F_\tau$ to $F^s_\tau$ with $(F^s_\tau)^*(q,p)=(F^*_{\tau}(q,p)^2+s(A_\tau(q)p,p))^{1/2},s\in (-\epsilon,\epsilon)$ as \eqref{equ: variation Fs}, where $A_\tau(\cdot)$ is supported in a small open neighborhood $\pi(U_1)$ of $\sigma(l_{\sigma}/2)\in \mathrm{reg}(\sigma)$ with the form
\begin{equation}\label{equ: choice of tilde A on U1}
\tilde A_\tau(\sigma_{t_1}(t_2))=\begin{pmatrix}a_{1}(t_2)&a_{2}(t_2)t_1\\a_{2}(t_2)t_1& a_{3}(t_2)t_1^2\end{pmatrix},\quad a_{1,2,3}\in C^\infty_0((l_\sigma/2-\epsilon,l_\sigma/2+\epsilon),\mathbb{R}),
\end{equation}
under the representation (\ref{equ: local representation of tilde A}). By Lemma \ref{lem: differential of P_gamma(f)} and the flexibility of $a_{1,2,3}$, we know that the differentiation of linear Poincar\'e map $P_{\sigma}(\cdot)$ is surjective at $0$. Therefore, one can perturb the family $F_{[0,1]^N}$ to $F_{[0,1]^N}^s$ within $\mathcal{F}^N_{F,C^r}$ near $F_\tau$ such that $\sigma\in\mathrm{CG}^{F^s_\tau}$ and $P_\sigma$ varies in an arbitrary direction in $\mathrm{Sp}(2)$. Therefore, $T_{P_\sigma}\mathrm{Im}\Gamma$ includes the whole fiber direction in $T_{P_\sigma}(\mathcal E|_{\sigma})$. Then we have
$$T_{P_\sigma}\mathrm{Im}\Gamma+T_{P_\sigma}\mathcal{E}_\theta=T_{P_\sigma}\mathrm{Im}\Gamma+T_{P_\sigma}\mathrm{CG}_p(\mathcal{F}^N_{F,C^r})=T_{P_\sigma}\mathcal{E}.$$
Hence, $\Gamma$ is transverse to $\mathcal{E}_{\theta}$ in $\mathcal E$ for every $\theta\in(0,\pi)\cup(\pi,2\pi)$. Similarly, $\Gamma$ is also transverse to $\mathcal{E}_{\pm1, b}$ and $\mathcal{E}_a$ in $\mathcal E$ for every $b\in\{0,\pm1\}$ and $a\in\mathbb{R}\setminus\{0,\pm 1\}$.

Moreover, whenever $F_\tau$ belongs to $\mathcal{F}^s(\Sigma)$ or $g_\tau=F_\tau^2\in\mathcal{M}(\Sigma)$, the previous variation $F^s_\tau$ is also in $\mathcal{F}^s(\Sigma)$ or $g_\tau^s=(F^s_\tau)^2\in\mathcal{M}(\Sigma),\forall s\in(-\epsilon,\epsilon)$. Based on Lemma \ref{lem: differential of P_gamma(f)} and the preceding arguments, the surjectivity of $dP_\sigma(\cdot)$ remains valid. Therefore, we maintain the transversality between $\Gamma^s$ and $\mathcal{E}^s_*$ in $\mathcal E^s$, as well as $\Gamma^R$ and $\mathcal{E}^R_*$ in $\mathcal E^R$. Then Lemma \ref{claim 2} holds.
\end{proof}
According to Lemma \ref{claim 2} and Lemma \ref{lem: Appendix Sard-Smale theorem 2}, we have $\mathrm{CG}_{p,*}(\mathcal{F}^N_{F,C^r})=\Gamma^{-1}(\mathcal{E}_{*})$ is a $C^{r-1}$-submanifold of $\mathrm{CG}_{p}(\mathcal{F}^N_{F,C^r})$ with codimension $\geq 1$. Since $P_1$ has index $N+1$, see \eqref{equ: def of P1}, the restricted map
$$P_{1,*}:\mathrm{CG}_{p,*}(\mathcal{F}^N_{F,C^r})\rightarrow \mathcal{F}^N_{F,C^r}\ \big((F_{[0,1]^N},\tau,\sigma)\mapsto F_{[0,1]^N}\big),\quad *\in\{(\pm 1,b),a,\theta\},$$
is $C^{r-1}$ Fredholm with index $\mathrm{Ind}P_{1,*}\leq N$.
Choose an $r\geq N+1$. Using Sard-Smale theorem, see Lemma \ref{lem: Appendix Sard-Smale theorem}, we see that the regular values of $P_{1,*}$ form a residual subset for every $*\in\{(\pm 1,b),a,\theta\}$. Therefore, the numerable intersection
\begin{equation}\label{equ: nondegenerate metrics}
\mathcal{F}^{N,\mathbf{b}}_{F,C^r}:=\bigcap_{b\in\{0,\pm1\}}\mathrm{reg}P_{1,\pm 1,b}\cap\bigcap_{\theta\in \mathbb{Q}\pi}\mathrm{reg}P_{1,\theta}
\end{equation}
is also a residual set in $\mathcal{F}^N_{F,C^r}$.

Similarly, since $\Gamma^s$ is transverse to $\mathcal{E}^s_*$ in $\mathcal{E}^s$ and $\Gamma^R$ is transverse to $\mathcal{E}^R_*$ in $\mathcal{E}^R$, we conclude that
$$\mathrm{CG}_{p,*}(\mathcal{F}^{s,N}_{F,C^r}):=\Gamma^{-1}(\mathcal{E}^s_*)\quad \text{and}\quad \mathrm{CG}_{p,*}(\mathcal{M}^{N}_{g,C^r}):=\Gamma^{-1}(\mathcal{E}^R_*)$$
are $C^{r-1}$-submanifold of $\mathrm{CG}_{p}(\mathcal{F}^{s,N}_{F,C^r})$ and $\mathrm{CG}_{p}(\mathcal{M}^N_{g,C^r})$. Moreover, since $P_1^s:\mathrm{CG}_{p}(\mathcal{F}^{s,N}_{F,C^r})\rightarrow \mathcal{F}^{s,N}_{F,C^r}$ and $P_1^R:\mathrm{CG}_{p}(\mathcal{M}^{N}_{g,C^r})\rightarrow \mathcal{M}^{N}_{g,C^r}$ remain index $N+1$, as described in \eqref{equ: def of P1}, the restricted maps
$$P^s_{1,*}:\mathrm{CG}_{p,*}(\mathcal{F}^{s,N}_{F,C^r})\rightarrow \mathcal{F}^{s,N}_{F,C^r}\quad \text{and}\quad P^R_{1,*}:\mathrm{CG}_{p,*}(\mathcal{M}^{N}_{g,C^r})\rightarrow \mathcal{M}^{N}_{g,C^r}$$
are $C^{r-1}$ Fredholm with index $\mathrm{Ind}P^s_{1,*}\leq N$ and $\mathrm{Ind}P^R_{1,*}\leq N$. By taking an $r\geq N+1$ and applying Lemma \ref{lem: Appendix Sard-Smale theorem} again, the regular value set of $P^s_{1,*}$ (resp. $P^R_{1,*}$) are of residual. By taking the intersection of numerical regular sets as (\ref{equ: nondegenerate metrics}), we obtain the following residual subsets:
$$\mathcal{F}^{s,N,\mathbf{b}}_{F,C^r}\subset\mathcal{F}^{s,N}_{F,C^r}, \quad  \mathcal{M}^{N,\mathbf{b}}_{g,C^r}\subset \mathcal{M}^{N}_{g,C^r}.$$
For $\mathcal{F}^{s,N,\mathbf{b}}_{F,C^r}$, $\mathcal{F}^{s,N,\mathbf{b}}_{F,C^r}$ and $\mathcal{F}^{s,N,\mathbf{b}}_{F,C^r}$, we have the following result.

\begin{lem}\label{claim 3}
For every $F_{[0,1]^N}\in \mathcal{F}^{N,\mathbf{b}}_{F,C^r}$ with $r\in \mathbb Z_+$ large enough, we have
$$\{\tau\in[0,1]^N: F_\tau\ \mathrm{is\ bumpy}\}\ \mathrm{has\ full\ measure\ in}\ [0,1]^N.$$
The same conclusion holds for $\mathcal{F}^{s,N,\mathbf{b}}_{F,C^r}$ and $\mathcal{M}^{N,\mathbf{b}}_{g,C^r}$.
\end{lem}
\begin{proof}
Let $F_{[0,1]^N}\in\mathcal{F}^{N,\mathbf{b}}_{F,C^r}$.  Since $F_{[0,1]^N}$ is a regular value of $P_{1,*}$, $dP_{1,*}$ is surjective at $F_{[0,1]^N}$, as described in section \ref{sec: Sard-smale theorem}. This implies that $\mathrm{Ind}P_{1,*}\leq N$ for every $*=(\pm 1, b)|_{b\in\{0,\pm 1\}}$ and $\theta\in\mathbb{Q}\pi$. Applying Lemma \ref{lem: Appendix Sard-Smale theorem 2}, we conclude that
$$\mathrm{CG}_{p,*}(F_{[0,1]^N}):=P_{1,*}^{-1}(F_{[0,1]^N}),\quad *=(\pm 1, b)|_{b\in\{0,\pm 1\}},\theta\in\mathbb{Q}\pi,$$
are $C^{r-1}$ submanifolds of $\mathrm{CG}_{p,*}(\mathcal{F}^N_{F,C^r})$ with dimension $\leq N$. Consider the projection
$$P_{2,*}(F_{[0,1]^N}):\mathrm{CG}_{p,*}(F_{[0,1]^N})\rightarrow [0,1]^N\ \big((F_{[0,1]^N},\tau,\sigma)\mapsto \tau\big).$$
There is a natural time-shifting $\mathbb R/l_\sigma\mathbb Z$-action
$$c:\mathrm{CG}_{p,*}(F_{[0,1]^N})\rightarrow \mathrm{CG}_{p,*}(F_{[0,1]^N})\ \big((F_{[0,1]^N},\tau,\sigma)\mapsto (F_{[0,1]^N},\tau,c\cdot\sigma)\big),$$
which is invariant under the projection $P_{2,*}(F_{[0,1]^N})$ for every $*\in\{(\pm 1, b)|_{b\in\{0,\pm 1\}},\theta\in\mathbb{Q}\pi\}$. Then $P_{2,*}(F_{[0,1]^N})$ is not immersed anywhere. Choose $r\in \mathbb Z_+$ large enough then applying Sard theorem (Lemma \ref{lem: Appendix Sard theorem}), we conclude that the critical values of $P_{2,*}(F_{[0,1]^N})$ form a measure zero subset in $[0,1]^N$. Since $dP_{2,*}(F_{[0,1]^N})$ never surjective due to the time translation, for every $\tau$ in the regular set of $P_{2,*}(F_{[0,1]^N})$, we must have $\tau\notin \mathrm{Im}P_{2,*}(F_{[0,1]^N})$. Therefore, we obtain a full measure subset
$$\Omega(F_{[0,1]^N}):= \bigcap_{*\in\{(\pm1,b)|_{b\in\{0,\pm1\}}, \theta\in\mathbb{Q}\pi\}}\mathrm{reg}P_{2,*}(F_{[0,1]^N})\subset [0,1]^N.$$
For each $t\in \Omega(F_{[0,1]^N})$,  $F_\tau$ is bumpy, i.e.
$$\mathrm{CG}^{F_\tau}\cap\mathrm{CG}_{p,*}(F_{[0,1]^N})=\emptyset,\quad *\in\{(\pm 1, b)|_{b\in\{0,\pm1\}},\theta\in\mathbb{Q}\pi\}.$$
Hence, the proof for $\mathcal{F}^{N,\mathbf{b}}_{F,C^r}$ is done. By following the similar arguments above, the same conclusion holds for $F_{[0,1]^N}\in\mathcal{F}^{s,N,\mathbf{b}}_{F,C^r}$ or $F_{[0,1]^N}\in\mathcal{M}^{N,\mathbf{b}}_{g,C^r}$. Hence, this lemma holds.
\end{proof}
\begin{proof}[Proof of Theorem \ref{lem: genericity of nondegenercy 2}]
Let $\mu$ be the Lebesgue measure in $[0,1]^N$. For any $L\gg 1$, we define
\begin{equation*}
\mathcal{F}^{N,\mathbf{b}}_{F,L}:=\left\{F_{[0,1]^N}\in \mathcal{F}^N_{F,C^\infty}:\ \mu\big\{\tau: \sigma\in\mathrm{CG}_p^{F_\tau}\ \mathrm{is\ nondegenerate\ if}\  l_\sigma\leq  L\big\}>1-L^{-1}\right\}.
\end{equation*}
We first establish that $\mathcal{F}^{N,\mathbf{b}}_{F,L}$ is open and dense in $\mathcal{F}^N_{F,C^\infty}$. Recall $\mathcal{F}^{N,\mathbf{b}}_{F,C^r}$ is a residual set in $\mathcal{F}^{N}_{F,C^\infty}$ for every $r>0$ sufficiently large. Then the countable intersection $\mathcal{F}^{N,\mathbf{b}}_{F,C^\infty}:=\bigcap_{r\geq 1}\mathcal{F}^{N,\mathbf{b}}_{F,C^r}$ becomes a dense subset in $\mathcal{F}^{N}_{F,C^\infty}$. By Lemma \ref{claim 3}, we know that $\mathcal{F}^{N,\mathbf{b}}_{F,C^\infty}\subset \mathcal{F}^{N,\mathbf{b}}_{F,L},\forall L\gg 1$ and this implies the denseness. To see the openness, it is sufficient to show that for every $F_{[0,1]^N}\in\mathcal{F}^{N,\mathbf{b}}_{F,L}$, all closed geodesics in $\mathrm{CG}_p^{F_\tau}$ with length $\leq L$ remains nondegenerate under a small perturbations of $F_{[0,1]^N}$. Indeed, the nondegenerate closed geodesics $\sigma\in \mathrm{CG}_p^{F_\tau}$ with $l_{F_\tau}(\sigma)\leq L$ are finitely many and isolated up to the canonical time translation. Under a sufficiently small perturbation of $F_\tau$, the nondegenercy of the closed geodesics in $l_{F_\tau}^{-1}(0,L]$ still remain. Moreover, there is no closed geodesic with length $\leq L$ arises, since the new created closed geodesics only have large periods. Hence, this implies the openness. Finally, the countable intersection $F^{N,\mathbf{b}}_{F,\infty}:=\bigcap_{L\geq 1}\mathcal{F}^{N,\mathbf{b}}_{F,L}$ becomes a residual set in $F^{N}_{F,C^\infty}$.

The proof for a generic $F_{[0,1]^N}\in \mathcal{F}^{s,N}_{F,C^\infty}$ or $g_{[0,1]^N}\in \mathcal{M}^{N}_{g,C^\infty}$ is analogous, as the aforementioned arguments and Lemma \ref{claim 3} hold in both $\mathcal{F}^{s,N,\mathbf{b}}_{F,C^r}$ and $\mathcal{M}^{N,\mathbf{b}}_{g,C^r}$. The proof is now completed.
\end{proof}

\section{Proof of main theorems} \label{sec: proof of main theorems}
In this section, we first prove Theorems \ref{thm: finsler} and \ref{thm: riemann}, which mainly concern symmetric Finsler and Riemannian metrics. These proofs combine the properties of ECH spectral invariants introduced in section \ref{sub: c_sigma of ECH}, together with Lemma \ref{lem: genericity of nondegenercy 1} and Theorem \ref{lem: genericity of nondegenercy 2} in section \ref{sec: generic nondegenercy of Rimannian metrics}. Consequently, we conclude Theorems \ref{thm: equidis finsler} and \ref{thm: equidis riemann}, by using Theorems \ref{thm: finsler} and \ref{thm: riemann}, along with the method of approximation and selection in \cite{MNS2017}. In conclusion, we establish the existence of an equidistributed sequence of closed geodesics for generic Finsler, symmetric Finsler and Riemannian metrics on every closed surface.

To establish the proof of Theorems \ref{thm: finsler} and \ref{thm: riemann}, we first introduce a result for Lipschitz functions by Marques, Neves and Song in \cite{MNS2017}, which has been applied in \cite{Irie2018}.
\begin{lem}[Marques-Neves-Song \cite{MNS2017}] \label{lem: lipschitz func}
Given a full measure subset $\mathcal{A}\subset[0,1]^N$. For every $\delta>0$ and $N\in\mathbb{N}_+$, there exists an $\epsilon=\epsilon(\delta,N)>0$ such that for every Lipschitz function $f:[0,1]^N\rightarrow \mathbb{R}$ with $\max f-\min f\leq 2\epsilon$, there exist $N+1$ many sequences $\{\tau^{j,m}\}_{m\geq 1}\subset \mathcal A$ satisfying
\begin{itemize}
\item[(1)] $\{\tau^{j,m}\}_{m\geq 1}$ converges to a point $\tau^\infty\in(0,1)^N$ for every $j=1,\cdots,N+1$.
\item[(2)] $f$ is differentiable at $\tau^{j,m}$ for every $m\geq 1$ and $j=1,\cdots,N+1$.
\item[(3)] $\{\nabla f(\tau^{j,m})\}_{m\geq 1}$ converges to some $v_j$ for every $j$, which satisfies
$$d_{\mathbb{R}^N}(0,\mathrm{conv}(v_1,\cdots,v_{N+1}))<\delta.$$
\end{itemize}
where $d_{\mathbb{R}^N}$ is the Euclidean distance in $\mathbb{R}^N$.
\end{lem}
Take a sequence of $\{\psi_i\}_{i\geq 1}$ in $C^\infty(\Sigma,\mathbb{R})$, which is a $C^0$-dense. We introduce the following lemma.
\begin{lem}\label{claim 4}
For any $N,\epsilon^{-1}>0$ sufficiently large, $F\in \mathcal F(\Sigma)$ and any open neighborhood $\mathcal{U}\subset \mathcal F(\Sigma)$ of $F$, there exists an $F'\in\mathcal{U}$ and a nondegenerate Reeb $1$-current $C\in \mathcal L(S^*_{F'}\Sigma,\lambda_{F'})$ so that
\begin{equation}\label{equ: condition of C}
\left|C(\psi_i\lambda_{F'})-\int_{S^*_{F'}\Sigma}\psi_i\lambda_{F'}\wedge d\lambda_{F'}\right|<\epsilon,\quad \forall i=1,\cdots,N.
\end{equation}
A similar statement holds for $F\in \mathcal F^s(\Sigma)$ or $\mathcal M(\Sigma)$ and open subset $\mathcal{U}\subset \mathcal F^s(\Sigma)$ or $\mathcal M(\Sigma)$.
\end{lem}
\begin{proof}
By Proposition \ref{prop: properties of Finsler geodesic and Reeb flows}-(1), $\mathcal{U}$ can be regarded as a neighborhood of $\lambda_{F}\in \Lambda^{Fin}_{F}$ for some $F\in \mathcal F(\Sigma)$. Let $\{\psi_i\}_{i\geq 1}$ be a $C^0$-dense sequence in $C^\infty(\Sigma,\mathbb{R})\subset C^\infty(S^*_{F}\Sigma,\mathbb{R})$. Let $c_{\varsigma_k}(S^*_{F}\Sigma,\lambda_{F})$ be the spectral invariant, see (\ref{equ: spectral invariant}), where $\{\varsigma_k\}_{k\geq1}$ is obtained from Proposition \ref{prop: properties of spectral invariant}-(5). Define
$$\gamma_k(\lambda_{F})=\frac{c_{\varsigma_k}(S^*_{F}\Sigma,\lambda_{F})}{\sqrt{2k}}
-\sqrt{\mathrm{vol}(S^*_{F}\Sigma,\lambda_{F})},\quad \forall k\geq 1.$$
as stated in Proposition \ref{prop: lipschitz property}. Choose a smooth $[0,1]^N$-family of contact forms in $\Lambda^{Fin}_{F}$ as
\begin{equation}\label{equ: formula of lambda tau}
\lambda_\tau=e^{f_\tau}\lambda_{F},\quad f_\tau=\frac{1}{\hat c}\sum_{i=1}^N\tau_i\psi_i\in C^\infty(\Sigma,\mathbb{R}),\quad \forall \tau\in[0,1]^N,\quad \hat c\gg0.
\end{equation}
Define $F_{[0,1]^N}\in\mathcal{F}^N_{F,C^\infty}$ as
$F_{\tau}:=e^{f_\tau}F$. From Proposition \ref{prop: properties of Finsler geodesic and Reeb flows}-(2), we have $(S^*_{F}\Sigma,\lambda_\tau)\cong(S^*_{F_\tau}\Sigma,\lambda_{F_\tau})$. Choose $\hat c\gg 0$ large enough so that $\lambda_{[0,1]^N}\subset\mathcal{U}$. According to Lemma \ref{lem: genericity of nondegenercy 2}, for every $\check c>0$ arbitrary small, there exists a perturbed $[0,1]^N$-family $\widetilde \lambda_{[0,1]^N}\subset \Lambda^{Fin}_{F}$ satisfying
\begin{itemize}
\item[(a)] $\mu\{\tau\in[0,1]^N:\widetilde \lambda_\tau\ \mathrm{is\ nondegenerate}\}=1$,
\item[(b)] $\|(\hat c/\widetilde\lambda_\tau)
    \partial_{\tau_i} \widetilde \lambda_\tau-\psi_i\|_{C^0}<\check c$.
\end{itemize}
Note that (b) follows from the relation $\partial_{\tau_i} \lambda_\tau=(\partial_{\tau_i} f_\tau)\lambda_\tau=(\psi_i/\hat c)\lambda_\tau$ and the fact that $\|\lambda_\tau-\widetilde \lambda_{\tau}\|_{C^\infty}$ admits an arbitrarily small uniform upper bound in $\tau$. By Proposition \ref{prop: lipschitz property}, we see that $f\mapsto \gamma_k(e^{f}\lambda_F)$ is a Lipschitz function in $C^0(S^*_{F}\Sigma,[-1,1])$. After expressing $\widetilde \lambda_\tau=e^{\widetilde f_\tau}\lambda_{F}$ on $\Lambda^{Fin}_{F}$, we obtain
\begin{equation}\label{equ: gamma k is lipschitz}
\big|\gamma_k(\widetilde\lambda_{\tau_1})- \gamma_k(\widetilde \lambda_{\tau_2})\big|\leq C\big\|\widetilde f_{\tau_1}-\widetilde f_{\tau_2}\big\|\leq C\max_{i,\tau}\big\|\partial_{\tau_i} \widetilde f_\tau\big\|_{C^0}\big|\tau_1-\tau_2\big|,\quad \forall k\geq 1,
\end{equation}
where $C$ only depends on $\lambda_{F}$. Therefore, the composition $\widetilde \gamma_k: [0,1]^N\rightarrow \mathbb{R}\ (\tau\mapsto \gamma_k(\widetilde \lambda_\tau))$ is Lipschitz for every $k\geq 1$. According to Proposition \ref{prop: properties of spectral invariant}-(5), we have $\displaystyle \lim_{k\rightarrow +\infty}\|\widetilde \gamma_k\|_{C^0}=0$.

Let $\widetilde{\mathcal{S}}$ be the full measure subset of $[0,1]^N$ in (a). By Lemma $\ref{lem: lipschitz func}$, we obtain an $\epsilon_0=\epsilon_0(\check c,N)>0$ such that for every $k>0$ sufficiently large so that $\max \widetilde\gamma_k-\min \widetilde\gamma_k\leq 2\epsilon_0$, there exist $N+1$ sequences $\{\tau^{j,m}\}_{m\geq 1},j=1,\cdots,N+1$, in $\widetilde{\mathcal{S}}$ such that
\begin{itemize}
\item[(c$_1$)] $\{\tau^{j,m}\}_{m\geq 1}$ converges to a point $\tau^\infty\in(0,1)^N$ for every $j=1,\cdots,N+1$.
\item[(c$_2$)] $\widetilde \gamma_k$ is differentiable at $\tau^{j,m}$ for every $m\geq 1$ and $j=1,\cdots,N+1$.
\item[(c$_3$)] $\{\nabla \widetilde \gamma_k(\tau^{j,m})\}_{m\geq 1}$ converges to some $v_j\in\mathbb{R}^N$ for every $j$, which satisfies $$d_{\mathbb{R}^N}\left(0,\mathrm{conv}(v_1,\cdots,v_{N+1})\right)<\check c.$$
\end{itemize}
Denote $\mathrm{vol}(\lambda):=\mathrm{vol}(S^*_{F}\Sigma,\lambda)$. From (c$_3$), we obtain a tuple of nonegative coefficients $\{a_j\}_{j=1}^{N+1}$ with $\sum_{j=1}^{N+1} a_j=1$ so that for every $i=1,\cdots,N$ and $m$ large enough, we compute
\begin{equation}\label{equ: linear combination smaller than check c}
\begin{aligned}
\check c&>\sum_{j=1}^{N+1} a_j \partial_{\tau_i} \widetilde\gamma_k(\tau^{j,m})=\sum_{j=1}^{N+1} a_j\left(\frac{\partial c_{\varsigma_k}\big(S^*_{F}\Sigma,\widetilde\lambda_{\tau^{j,m}}\big)}{\sqrt{2k}\ \partial \tau_i}-\frac{\int_{S^*_{F}\Sigma}\partial_{\tau_i} \widetilde\lambda_{\tau^{j,m}}\wedge d\widetilde \lambda_{\tau^{j,m}}}{\sqrt{\mathrm{vol}\big(\widetilde \lambda_{\tau^{j,m}}\big)}}\right)\\
&\geq \sum_{j=1}^{N+1} a_j\left(\frac{C_{j,m}\big((\psi_i/\hat c)\widetilde\lambda_{\tau^{j,m}}\big)}{\sqrt{2k}}-\frac{\int_{S^*_{F}\Sigma}(\psi_i/\hat c)\widetilde\lambda_{\tau^{j,m}}\wedge d\widetilde\lambda_{\tau^{j,m}}}{\sqrt{\mathrm{vol}\big(\widetilde\lambda_{\tau^{j,m}}\big)}}\right)-\frac{D\check c}{\hat c},
\end{aligned}
\end{equation}
where $C_{j,m}\in \mathcal L(S^*_{F}\Sigma,\widetilde \lambda_{\tau^{j,m}})$ is the Reeb $1$-current obtained by Proposition \ref{prop: regularity of spectral invariant}-(1), i.e. $C_{j,m}(\widetilde \lambda_{\tau^{j,m}})=c_{\varsigma_k}(S^*_{F}\Sigma,\widetilde \lambda_{\tau^{j,m}})$, and $D:=\max_{j,m,k}\left\{(2k)^{-\frac{1}{2}}C_{j,m}(\widetilde\lambda_{\tau^{j,m}})
+\mathrm{vol}(\widetilde \lambda_{\tau^{j,m}})^{\frac{1}{2}}\right\}.$ Note that the last inequality in (\ref{equ: linear combination smaller than check c}) is derived from (b), (c$_2$) and Proposition \ref{prop: regularity of spectral invariant}. Moreover, using Proposition \ref{prop: properties of spectral invariant}-(5) and (c$_1$), the maximum $D$ exists and only relies on $\widetilde \lambda_{[0,1]^N}$.
Therefore, by (\ref{equ: linear combination smaller than check c}), we have
\begin{equation}\label{equ: linear combination smaller than hat c}
\left|\sum_{j=1}^{{N+1}} a_j\left(\frac{C_{j,m}\big(\psi_i\widetilde\lambda_{\tau^{j,m}}\big)}{\sqrt{2k}}-\frac{\int_{S^*_{F}\Sigma}\psi_i\widetilde\lambda_{\tau^{j,m}}\wedge d\widetilde \lambda_{\tau^{j,m}}}{\sqrt{\mathrm{vol}\big(\widetilde \lambda_{\tau^{j,m}}\big)}}\right)\right|< \check c(\hat c+D).
\end{equation}
Since $\widetilde \lambda_{\tau^{j,m}}$ converges to $\widetilde\lambda_{\tau^\infty}$ as $m\rightarrow +\infty$ for every $j$, and the Reeb orbits in $\{C_{j,m}\}_{m\geq1}$ exhibit bounded periods, a limiting current can be obtained along a subsequence:
$$C_\infty:=\frac{1}{\sqrt{2k}}\sqrt{\mathrm{vol}\big(\widetilde \lambda_{\tau^\infty}\big)}\lim_{m\rightarrow +\infty}\sum_{j=1}^{N+1} a_jC_{j,m}\in \mathcal L(S^*_{F}\Sigma,\widetilde \lambda_{\tau^\infty}).$$
Let $F'=\widetilde F_{\tau^\infty}\in \mathcal U$. We see from Proposition \ref{prop: properties of Finsler geodesic and Reeb flows}-(1) that $(S^*_F\Sigma,\tilde \lambda_{\tau^\infty})$ is isomorphic to $(S^*_{F'}\Sigma,\lambda_{F'})$. Then $C_\infty$ can be regarded as a Reeb $1$-current in $\mathcal L(S^*_{F'}\Sigma,\lambda_{F'})$. It follows from (\ref{equ: linear combination smaller than hat c}) that
\begin{eqnarray*}
\left|C_\infty(\psi_i \lambda_{F'}) -\int_{S^*_{F'}\Sigma}\psi_i\lambda_{F'}\wedge d\lambda_{F'}\right|< \sqrt{\mathrm{vol}\big(S^*_{F'}\Sigma,\lambda_{F'}\big)}\cdot \check c(\hat c+D),\quad \forall i=1,\cdots,N.
\end{eqnarray*}
Therefore, by fixing $\hat c>0$ small enough so that $F_{[0,1]^N}\subset\mathcal{U}$, one can choose $\check c>0$ sufficiently small so that $\mathrm{vol}\big(S^*_{F'}\Sigma,\lambda_{F'}\big)^{1/2}\cdot \check c(\hat c+D)<\epsilon$ through a suitable perturbation $\widetilde F_{[0,1]^N}$ of $F_{[0,1]^N}$ in $\mathcal U$. Then \eqref{equ: condition of C} holds for $F'$ and $C_\infty$. Finally, using Lemma \ref{lem: genericity of nondegenercy 1}, we can further perturb $F'$ slightly, so that $C_{\infty}$ is a nondegenerate Reeb $1$-current in $\mathcal L(S^*_{F'}\Sigma, \lambda_{F'})$, which still  satisfies \eqref{equ: condition of C}. The proof of this lemma holds for $\mathcal U\subset \mathcal F(\Sigma)$.

The proof for $\mathcal{U}\subset\mathcal F^s(\Sigma)$ or $\mathcal M(\Sigma)$ is analogous. Firstly, the family (\ref{equ: formula of lambda tau}) is a fiberwise scaling of $\lambda_F$ on $S^*_F\Sigma$, which remains a family of contact forms in $\Lambda^{sFin}_F$ or $\Lambda^{Rie}_g$ if $F\in \mathcal F^s(\Sigma)$ or $\mathcal M(\Sigma)$. Since Lemma~\ref{lem: genericity of nondegenercy 1} and Theorem \ref{lem: genericity of nondegenercy 2} are applicable in $\mathcal F^s(\Sigma)$ and $\mathcal M(\Sigma)$, the remaining steps are still valid. Therefore, this lemma holds.
\end{proof}
\subsection{Proof of Theorem \ref{thm: finsler}}
Let $F\in \mathcal F^s(\Sigma)$ and let $\Lambda^{sFin}_{N,\epsilon}$ denote the collection of $\lambda\in \Lambda^{sFin}_F$, which possesses a nondegenerate Reeb 1-current $C^N_{\lambda,\epsilon}\in \mathcal L(S^*_F\Sigma,\lambda)$ such that $\big|C^N_{\lambda,\epsilon}(\psi_i\lambda)-\int_{S^*_F\Sigma}\psi_i\lambda\wedge d\lambda\big|<\epsilon$ for every $i=1,\cdots,N$.

We first prove that $\Lambda^{sFin}_{N,\epsilon}$ is open and dense in $\Lambda^{sFin}_F$. The denseness directly follows from Lemma \ref{claim 4}. To establish the openness, we consider the nondegenerate Reeb 1-current $C^N_{\lambda_0,\epsilon}=\sum_{j=1}^ka_j\gamma_j$ of a given $\lambda_0\in\Lambda^{Fin}_{N,\epsilon}$. For every $\lambda\in \Lambda^{Fin}_F$ sufficiently close to $\lambda_0$, there exists a nondegenerate Reeb $1$-current $C^N_{\lambda,\epsilon}=\sum_{j=1}^ka_{j}\gamma_{j}(\lambda)\in \mathcal L(S^*_F\Sigma,\lambda)$ satisfying the same conditions and smooth in $\lambda$. This establishes the openness. Choose $\epsilon=N^{-1}$. Then the countable intersection
$$\Lambda^{sFin}_{F,0}:=\bigcap_{N\geq 1}\Lambda^{sFin}_{N,N^{-1}}$$
forms a residual set in $\Lambda^{sFin}_F$. For every $\lambda\in \Lambda^{sFin}_{F,0}$, there exists a nondegenerate sequence of Reeb $1$-current $\{C^N_{\lambda,N^{-1}}\}_{N\geq 1}\subset \mathcal L(S^*_F\Sigma,\lambda)$ so that
$$ \left|C^N_{\lambda,N^{-1}}(\psi_i\lambda)-\int_{S^*_F\Sigma}\psi_i\lambda\wedge d\lambda\right|\rightarrow 0,\quad \forall i=1,\cdots,N\quad \text{as}\quad N\rightarrow +\infty.$$
Since the sequence $\{\psi_i\}_{i\geq 1}$ is $C^0$-dense in $C^\infty(\Sigma,\mathbb{R})$, we can infer from its continuity that
$$\left|C^N_{\lambda,N^{-1}}(f\lambda)-\int_{S^*_F\Sigma}f\lambda\wedge d\lambda\right|\rightarrow 0,\quad f\in C^\infty(\Sigma,\mathbb{R}).$$
Finally, since $\Lambda^{sFin}_F$ is one-to-one corresponding to $\mathcal F^s(\Sigma)$, we complete the proof for generic $F\in \mathcal F^s(\Sigma)$. The proof for $\Lambda^{Fin}_F$ is entirely analogous to $\Lambda^{sFin}_F$, as Lemma \ref{claim 4} is applicable in the case of $\Lambda^{Fin}_F$. Consequently, Theorem \ref{thm: finsler} holds.

\subsection{Proof of Theorem \ref{thm: riemann}}
Let $g\in \mathcal M(\Sigma)$. Take a $C^0$-dense sequence of $\{\psi_i\}_{i\geq 1}$ in $C^\infty(\Sigma,\mathbb{R})\subset C^\infty(S^*_g\Sigma,\mathbb{R})$. For any open subset $\mathcal{U}\subset \mathcal M(\Sigma)$, we define a $[0,1]^N$-family of contact forms in $\Lambda^{Rie}_{g}$ as
\begin{equation}\label{equ: formula of lambda tau in Rie}
\lambda_\tau=e^{f_\tau}\lambda_{g},\quad f_\tau=\frac{1}{\hat c}\sum_{i=1}^N\tau_i\psi_i\in C^\infty(\Sigma,\mathbb{R}),\quad \forall \tau\in[0,1]^N.
\end{equation}
Due to Proposition \ref{prop: properties of Riemannian geodesic and Reeb flows}-(2), $\mathcal U$ can be regarded as an open subset of $\Lambda^{Rie}_g$. 
Take a $\hat c\gg 0$ large enough so that $\lambda_{[0,1]^N}\subset\mathcal{U}$. By Lemma \ref{claim 4}, for any $N,\epsilon^{-1}>0$ sufficiently large, there exists a $g'\in \mathcal{U}$ and a nondegenerate Reeb $1$-current $C^N_{\lambda,\epsilon}\in \mathcal L(S^*_{g'}\Sigma,\lambda_{g'})$ such that
\begin{equation}\label{equ: limit of Reeb currents with epsilon}
\left|C^N_{\lambda,\epsilon}(\psi_i\lambda)-\int_{S^*_{g'}\Sigma}\psi_i\lambda_{g'}\wedge d\lambda_{g'}\right|<\epsilon,\quad \forall i=1,\cdots,N.
\end{equation}
Define $\Lambda^{Rie}_{N,\epsilon}$ as the collection of $\lambda\in \Lambda^{Rie}_g$, which possesses a nondegenerate Reeb 1-current $C^N_{\lambda,\epsilon}\in \mathcal L(S^*_g\Sigma,\lambda)$ satisfying (\ref{equ: limit of Reeb currents with epsilon}). Similar to the proof of Theorem \ref{thm: finsler}, we have the denseness and openness of $\Lambda^{Rie}_{N,\epsilon}$ in $\Lambda^{Rie}_g$. Therefore, the countable intersection $\Lambda^{Rie}_{g,0}:=\cap_{N\geq 1}\Lambda^{Rie}_{N,N^{-1}}$ forms a residual subset in $\Lambda^{Rie}_g$. For every $\lambda\in \Lambda^{Rie}_{g,0}$, there is a sequence of nondegenerate Reeb $1$-currents $\{C^N_{\lambda,N^{-1}}\}_{N\geq 1}\subset \mathcal L(S^*_g\Sigma,\lambda)$ satisfying
$$ \left|C^N_{\lambda,N^{-1}}(\psi_i\lambda)-\int_{S^*_g\Sigma}\psi_i\lambda\wedge d\lambda\right|\rightarrow 0,\quad \forall i=1,\cdots,N,\quad \text{as}\quad N\rightarrow +\infty.$$
Since $\{\psi_i\}_{i\geq 1}$ is $C^0$-dense in $C^\infty(\Sigma,\mathbb{R})$, then it follows from continuity that
$$\bigg|C^N_{\lambda,N^{-1}}(f\lambda)-\int_{S^*_g\Sigma}f\lambda\wedge d\lambda\bigg|\rightarrow 0,\quad \forall f\in C^\infty(\Sigma,\mathbb{R}).$$
Finally, since $\Lambda^{Rie}_g$ is one-to-one corresponding to $\mathcal M(\Sigma)$, we complete the proof of Theorem \ref{thm: riemann}.

\subsection{Proof of Theorem \ref{thm: equidis finsler}}
Consider a generic element $F\in \mathcal F^s(\Sigma)$ so that there exists a sequence of Finsler $1$-currents $\{C_i\}_{i\geq 1}\in \mathcal L^F$ satisfying (\ref{equ: limit of Reeb currents Finsler}) from Theorem \ref{thm: finsler}. Rewrite $C_i=\sum_{j=1}^{m_i}a_{i,j}\gamma_{\sigma_{i,j}}\in \mathcal L(S^*_{F}\Sigma,\lambda_F)$ as a Reeb $1$-current with $a_{i,j}\geq 0$ and $\sigma_{i,j}\in \mathrm{CG}^{F}$. Choose an approximating sequence
$\widehat C_i=n_i^{-1}\sum_{j=1}^{m_i}r_{i,j}\gamma_{\sigma_{i,j}}$ ensuring the validity of \eqref{equ: limit of Reeb currents Finsler} with $n_i,r_{i,j}\in\mathbb{Z}$ for every $i\geq 1,j=1,\cdots,n_i$. If $f=1$, we obtain from \eqref{equ: limit of Reeb currents Finsler} that
$$\lim_{i\rightarrow +\infty}\frac{1}{n_i}\sum_{j=1}^{m_i}r_{i,j}l_{F}(\sigma_{i,j})
=2\pi\mathrm{vol}_{F}\Sigma,$$
where $l_{F}(\sigma_{i,j})=\gamma_{\sigma_{i,j}}(\lambda_F)$ and $2\pi\mathrm{vol}_{F}\Sigma=\mathrm{vol}(S^*_F\Sigma,\lambda_F)$ are used. Moreover, by Proposition \ref{prop: properties of Finsler geodesic and Reeb flows}-(5) and (7), we have
$\gamma_{\sigma_{i,j}}(f\lambda)=\sigma_{i,j}(f\beta_{F})$ and $\int_{S^*_{F}\Sigma}f\lambda_F\wedge d\lambda_F=2\pi\int_\Sigma fd\mathrm{vol}_{F}$ for every $f\in C^\infty(\Sigma,\mathbb{R}).$
Then \eqref{equ: limit of Reeb currents Finsler} further implies that
\begin{eqnarray}\label{equ: limit of closed geodesics 2}
\lim_{i\rightarrow +\infty}\frac{\sum_{j=1}^{m_i}r_{i,j}\sigma_{i,j}(f\beta_{F})}{\sum_{j=1}^{m_i}r_{i,j}l_{F}(\sigma_{i,j})}=\frac{ d\mathrm{vol}_{F}(f)}{\mathrm{vol}_{F}\Sigma},\quad \forall f\in C^\infty(\Sigma,\mathbb{R}).
\end{eqnarray}
Denote $\mathcal{A}_N:=\sum_{i=1}^N n_iA_i$, $\mathcal{B}_N:=\sum_{i=1}^N n_iB_i$ and $C(f):=\frac{d\mathrm{vol}_{F}(f)}{\mathrm{vol}_{F}\Sigma}$, where $A_i=\sum_{j=1}^{m_i}r_{i,j}\sigma_{i,j}$, $B_i=\sum_{j=1}^{m_i}r_{i,j}l_{F}(\sigma_{i,j})>0$ and $n_i\in\mathbb{N}_+$. The following claim holds.
\begin{claim}\label{claim 5}
If the sequence $\{n_i\}$ satisfies $n_N\frac{B_{N}}{\mathcal{B}_{N-1}}\rightarrow +\infty$ as $N\rightarrow +\infty$, then
$\underset{N\rightarrow +\infty}{\lim}\frac{\mathcal{A}_N(f)}{\mathcal{B}_N}=C(f)$, $\forall f\in C^\infty(\Sigma,\mathbb{R})$.
\end{claim}
Indeed, denote $\mathcal{A}_{N_0+1,N}=\mathcal{A}_{N}-\mathcal{A}_{N_0}$ and $\mathcal{B}_{N_0+1,N}=\mathcal{B}_{N}-\mathcal{B}_{N_0}$, we have
\begin{eqnarray*}
\frac{\mathcal{A}_N(f)}{\mathcal{B}_N}&=&\frac{\mathcal{A}_{N_0}(f)+\mathcal{A}_{N_0+1,N}(f)}{\mathcal{B}_{N_0}+\mathcal{B}_{N_0+1,N}}\\
&=&\frac{\mathcal{A}_{N_0}(f)}{\mathcal{B}_{N_0}}\frac{\mathcal{B}_{N_0}}{\mathcal{B}_{N_0}+\mathcal{B}_{N_0+1,N}}
+\frac{\mathcal{A}_{N_0+1,N}(f)}{\mathcal{B}_{N_0+1,N}}\frac{\mathcal{B}_{N_0+1,N}}{\mathcal{B}_{N_0}+\mathcal{B}_{N_0+1,N}}\\
&\leq&\frac{\mathcal{A}_{N_0}(f)}{\mathcal{B}_{N_0}}\frac{\mathcal{B}_{N_0}}{\mathcal{B}_{N_0}+\mathcal{B}_{N_0+1,N}}
+(\epsilon_{N_0}(f)+C(f))\frac{\mathcal{B}_{N_0+1,N}}{\mathcal{B}_{N_0}+\mathcal{B}_{N_0+1,N}}\\
&\leq& \frac{\mathcal{A}_{N_0}(f)}{\mathcal{B}_{N_0}}+\epsilon_{N_0}(f)+C(f)=:D_{N_0}(f),
\end{eqnarray*}
where $|\frac{A_N(f)}{B_N}-C(f)|<\epsilon_{N_0}(f)$ for any $N> N_0$ sufficiently large according to (\ref{equ: limit of closed geodesics 2}). Then $\frac{\mathcal{A}_N(f)}{\mathcal{B}_N}$ is bounded and
\begin{eqnarray*}
\frac{\mathcal{A}_N(f)}{\mathcal{B}_N}&=&\frac{\mathcal{A}_{N-1}(f)+n_NA_N(f)}{\mathcal{B}_{N-1}+n_NB_N}\\
&=&\frac{\mathcal{A}_{N-1}(f)}{\mathcal{B}_{N-1}}\frac{1}{1+n_N\frac{B_N}{\mathcal{B}_{N-1}}}
+\frac{A_N(f)}{B_N}\frac{n_N\frac{B_N}{\mathcal{B}_{N-1}}}{1+n_N\frac{B_N}{\mathcal{B}_{N-1}}}\\
&\rightarrow&\lim_{N\rightarrow +\infty}\frac{A_N(f)}{B_N}=C(f).
\end{eqnarray*}
Hence, this claim holds.

Finally, Claim \ref{claim 5} implies (\ref{equ: limit of Finlser closed geodesics}), in which the sequence $\{\sigma_i\}_{i\geq 1}$ can be chosen as $\bigcup_{i\geq 1}\{\sigma_{i,j}^{n_ir_{i,j}}\}_{j=1}^{m_i}$ in $\mathrm{CG}^{F}$. Here $\sigma^r$ denotes the $r$-iteration of $\sigma$. The proof for generic $F\in \mathcal F(\Sigma)$ is entirely the same. Hence, the proof is completed.

\subsection{Proof of Theorem \ref{thm: equidis riemann}}
Consider a generic element $g\in \mathcal M(\Sigma)$ so that there exists a sequence of Reeb $1$-currents $\{C_i\}_{i\geq 1}\in \mathcal{L}^g$ satisfying \eqref{equ: limit of Reeb currents Riemannian} from Theorem \ref{thm: riemann}. Rephrase $C_i=\sum_{j=1}^{m_i}a_{i,j}\gamma_{\sigma_{i,j}}\in \mathcal L(S^*_{g}\Sigma,\lambda_g)$ as a Reeb $1$-current and then replace $C_i$ by an approximating sequence of $\mathbb{Q}$-coefficient Reeb $1$-current $\widehat C_i:=n_i^{-1}\sum_{j=1}^{m_i}r_{i,j}\gamma_{\sigma_{i,j}}$ ensuring \eqref{equ: limit of Reeb currents Riemannian} still holds. If $f\equiv 1$, by \eqref{equ: limit of Reeb currents Riemannian}, we have
$$\lim_{i\rightarrow +\infty}\frac{1}{n_i}\sum_{j=1}^{m_i}r_{i,j}l_{g}(\sigma_{i,j})=2\pi\mathrm{vol}_{g}\Sigma,$$
where the identities $l_{g}(\sigma_{i,j})=\sigma_{i,j}(\lambda_g)$ and $2\pi\mathrm{vol}_{g}\Sigma=\mathrm{vol}(S^*_g\Sigma,\lambda_g)$ are used. Let $s_g$ denote the arc length parameter along every $\sigma\in\mathrm{CG}^g$. According to Proposition \ref{prop: properties of Riemannian geodesic and Reeb flows}-(5) and (7), we have
$$\sigma_{i,j}(f\lambda_g)=\sigma_{i,j}(fds_g),\quad \int_{S^*_g\Sigma}f\lambda_g\wedge d\lambda_g=2\pi\int_\Sigma fd\mathrm{vol}_{g},\quad \forall f\in C^\infty(\Sigma,\mathbb{R}).$$
Then (\ref{equ: limit of Reeb currents Riemannian}) can be further rephrased as
\begin{eqnarray}\label{equ: limit of closed geodesics 3}
\lim_{i\rightarrow +\infty}\frac{\sum_{j=1}^{m_i}r_{i,j}\sigma_{i,j}(fds_{g})}{\sum_{j=1}^{m_i}r_{i,j}l_{g}(\sigma_{i,j})}=\frac{\int_\Sigma f d\mathrm{vol}_{g}}{\mathrm{vol}_{g}\Sigma},\quad \forall f\in C^\infty(\Sigma,\mathbb{R}).
\end{eqnarray}
Combining Claim \ref{claim 5} and \eqref{equ: limit of closed geodesics 3}, we directly obtain (\ref{equ: limit of Riemannian closed geodesics}), in which the sequence $\{\sigma_i\}_{i\geq 1}$ is selected as $\bigcup_{i\geq 1}\{\sigma_{i,j}^{n_ir_{i,j}}\}_{j=1}^{m_i}$ in $\mathrm{CG}^{g}$, similar to the proof of Theorem \ref{thm: equidis finsler}. Hence, the proof is completed.



\end{document}